\documentclass[12pt]{article}
\usepackage{amssymb,amsmath,enumerate, amsthm, graphicx}

\newtheorem{theorem}{Theorem}[section]
\newtheorem{lemma}[theorem]{Lemma}

\newtheorem{definition}[theorem]{Definition}

\newtheorem{observation}[theorem]{Observation}
\newtheorem{conjecture}[theorem]{Conjecture}
\newtheorem{claim}[theorem]{Claim}

 \newcommand{\eps}{\varepsilon}                       
      \renewcommand{\epsilon}{\varepsilon}  
       \renewcommand{\phi}{\varphi}

\title{Spanning Trees in Graphs of High Minimum Degree with a Universal Vertex I: \\An Asymptotic Result } 
\author{
Bruce Reed\footnote{University of Victoria. Research supported by NSERC.} 
\quad
Maya Stein\footnote{University of Chile, Research supported by ANID Regular Grant 1221905, by FAPESP-ANID Investigaci\'on Conjunta grant 2019/13364-7, and by ANID PIA CMM FB210005.}
}

\date{}
\begin{document}
\maketitle

\begin{abstract}
In this paper and a companion paper, we prove that, if $m$ is  
sufficiently large, every graph on $m+1$ vertices that has a universal vertex and minimum degree at least $\lfloor \frac{2m}{3} \rfloor$ contains each tree $T$ with $m$ edges as a subgraph.  
Our result confirms, for large $m$, an important special case of a recent conjecture by Havet, Reed, Stein, and Wood. 
The  present paper already contains an approximate   version of the result. 
\end{abstract}

\section{Introduction}
A recurring topic in extremal graph theory is the use of degree conditions (such as minimum/average degree bounds)  on a graph to prove that it contains certain subgraphs. One of the easiest classes of subgraphs for which this question  is not yet properly understood are trees. This is the focus of the present paper.

Clearly, any graph of minimum degree exceeding  $m-1$ contains a copy of each tree with $m$ edges: Just embed the root of the tree anywhere in the host graph, and greedily continue, always embedding vertices whose parents have  already been embedded. 
The bound on the minimum degree is sharp (see below).

Our paper is one of a large number which discuss
possible strengthenings  of the above observation by replacing the minimum degree condition with a different condition on the degrees of the host graph. One of these is the Loebl-Koml\'os-S\'os conjecture from 1995 (see~\cite{EFLS95}), which replaces the minimum degree with the median degree. This conjecture has attracted a fair amount of attention over the last decades, and has been settled asymptotically~\cite{cite:LKS-cut0, cite:LKS-cut1,cite:LKS-cut2, cite:LKS-cut3}. 

More famously, 
 Erd\H os and S\' os conjectured in 1963 that
every graph of average degree exceeding $m-1$  contains 
each tree with $m$ edges as a subgraph.  This conjecture would be best possible, since no $(m-1)$-regular graph contains the star $K_{1,m}$ as a subgraph. Alternatively, one can consider a graph that consists of several disjoint copies of the complete graph $K_{m}$; this graph has no connected $(m+1)$-vertex subgraph at all. Note that for both of these examples it does not matter whether we considered the average degree (as in the Erd\H os--S\' os conjecture) or the minimum degree (as in the observation above). 

The Erd\H os--S\' os conjecture poses an extremely interesting  question. It is trivial for stars, and it holds for paths by an old theorem of Erd\H os and Gallai~\cite{EG59}. It also holds when some additional assumptions on the host graph are made, see for instance~\cite{bradob, Haxell:TreeEmbeddings, sacwoz, kalaiBip}, and there are versions for bounded degree trees~\cite{BPS2, rohzon}. In the early 1990's, Ajtai, Koml\'os, Simonovits and Szemer\' edi  announced a proof of the Erd\H os--S\'os conjecture for sufficiently 
large $m$. 

It is well-known that every graph of average degree $>m$ has a subgraph of minimum degree $>\frac m2$. So,
  if it were true that every graph of minimum degree exceeding $\frac m2$ contained each tree on $m$ edges, then the Erd\H os--S\'os conjecture would immediately follow. Of course, the statement from the previous sentence is not true: It suffices to consider the examples given above. Still,  for bounded degree spanning trees an approximate version of the statement does hold. 
Koml\'os, Sark\"ozy and Szemer\'edi show in~\cite{KSS} that every  large enough $(m+1)$-vertex graph of minimum degree at least $(1+\delta)\frac m2$ contains each tree with $m$ edges whose maximum degree is bounded by $\frac{cn}{\log n}$, where the constant $c$  depends on $\delta$.  Variations of the bounds and the size of the tree are  given in~\cite{BPS, CLNS10}. However, the result from~\cite{KSS} is essentially best possible in the sense that    (even if  the minimum degree of the host graph is raised) it does not hold for trees of significantly larger maximum degree~\cite{KSS}. 


So, if we wish to find a condition that guarantees we can find {\em all} trees of a given size as subgraphs, only bounding the minimum degree is not enough.
Nevertheless, there can be at most one vertex of degree at least $\frac m2$ in any tree on $m+1$ vertices, and so, we might not need many vertices of large degree in the host graph. Therefore, it seems natural to try to pose a condition on both the minimum and the maximum degree of the host graph.

 The first conjecture of this type has been put forward recently by
Havet, Reed, Stein, and Wood~\cite{HRSW}.   They believe that a maximum degree of at least $m$ and a minimum degree of  at least $\lfloor \frac{2m}{ 3}  \rfloor$ is enough to embed all $m$-edge trees.

\begin{conjecture}[Havet, Reed, Stein, and Wood~\cite{HRSW}]
\label{conj1}
Let $m\in \mathbb N$. 
If a graph has maximum  degree at least $m$ and minimum degree at least $\lfloor \frac{2m}{ 3}  \rfloor$ then it contains every tree
with $m$ edges as a subgraph.
\end{conjecture}

The conjecture holds if the minimum degree condition is replaced by $(1-\gamma)m$, for a tiny but explicit\footnote{Namely, $\gamma=200^{-30}$.} constant $\gamma$, and it also holds if the maximum degree condition is replaced by a large function\footnote{Namely, $f(m)=(m+1)^{2m+6}+1$.} in $m$~\cite{HRSW}. 

Moreover, an  approximate version of the conjecture holds for bounded degree trees and dense host graphs~\cite{BPS}. Such an approximate version even holds for a generalised form of Conjecture \ref{conj1}, where the bound on the minimum degree is allowed to be any value between $\frac m2$ and $\frac{2m}3$, with the maximum degree obeying a corresponding bound between $2m$ and $m$ (see~\cite{BPS3} for details).

As further evidence for Conjecture~\ref{conj1} we shall prove, in this paper and its companion paper~\cite{brucemaya2},  that the conjecture holds for sufficiently large~$m$, under the additional assumption that the host graph has $m+1$ vertices, i.e.,  we are looking for a spanning tree. That is, building on the results from the present paper, we will show the following theorem in~\cite{brucemaya2}.

\begin{theorem} $\!\!${\rm\bf\cite{brucemaya2}}
\label{maint}  There is an $m_0\in\mathbb N$ such that for every $m \ge m_0$ 
every graph on $m+1$ vertices which has minimum degree at least $\lfloor\frac{2m}3\rfloor$ and  a universal vertex contains every tree
$T$ with $m$ edges as a subgraph.
\end{theorem}


Observe that Theorem~\ref{maint} is easy if $T$ has a vertex $t$ that is adjacent to a set $L$ of at  least $\lceil \frac{m}{3}\rceil$ leaves. Indeed, we  root $T$ at $t$, embed $t$  in the universal vertex~$v^*$ of $G$, greedily embed $T-L$,  and then embed $L$ in neighbours of~$v^*$. This is possible since $v^*$ is universal.  
It turns out that this approach can be extended if,
for a small positive number $\delta$, the tree $T$ contains a vertex adjacent to at least $\delta n$ leaves. 
Although  the greedy argument no longer works, we will be able to prove a result, namely Lemma~\ref{realhighleafdegreecase} below, which achieves the embedding of any tree~$T$ as above. This lemma will be crucial for the proof of Theorem~\ref{maint} in our companion paper~\cite{brucemaya2}.

\begin{lemma}
\label{realhighleafdegreecase}
For every $\delta>0$, there is an $m_\delta$ such that for any $m \ge m_\delta$  the following holds for every graph $G$  on $m+1$ vertices which has minimum degree at least $\lfloor\frac{2m}3\rfloor$ and a universal vertex. If $T$ is a tree with  $m$ edges, and  some  vertex of $T$ is  adjacent 
to at least~$\delta m$ leaves, then~$G$ contains~$T$.
\end{lemma}

 Also, the results from the present paper already imply an  asymptotic version of Theorem~\ref{maint}.

\begin{theorem}
\label{chighleafdegreecase}
For every $\delta>0$, there is an $m_\delta$ such that 
  for every $m\geq m_\delta$, every graph $G$ on $m+1$ vertices having minimum degree at least $\lfloor\frac{2m}3\rfloor$ and a universal vertex contains every tree $T$ having at most $(1-\delta) m$ edges.
\end{theorem}

Both
Theorem~\ref{chighleafdegreecase} and Lemma~\ref{realhighleafdegreecase} follow quickly from  Lemma~\ref{highleafdegreecase}, which is stated in Section~\ref{lemmas}. The  proof of Lemma~\ref{highleafdegreecase} occupies almost all the remainder of this paper, namely Sections~4--6, and a sketch of this proof is given in Section~\ref{ooover}. 

In the companion paper~\cite{brucemaya2}, we  prove the full Theorem~\ref{maint}, building on Lemma~\ref{realhighleafdegreecase} and on another auxiliary result, namely Lemma~\ref{maya1}, which is similar to Lemma~\ref{realhighleafdegreecase}.
Lemma~\ref{maya1} is  stated and proved in the last section of the present paper, Section \ref{sec:last}. 

\medskip

Let us end the introduction with a very short sketch of our methods of proof. A more detailed overview can be found in Section~\ref{ooover}.

Given a tree $T$ we wish to embed in the host graph $G$,  we first cut~$T$ into a constant number of connecting vertices, and a large number of very small subtrees. 
Applying regularity to $G$, we can ensure that all  small trees that are not just leaves can be embedded into matching structures we find in the reduced graph. This is more complicated than in  earlier work on tree embeddings using the regularity approach, as our  assumptions are too weak to force one matching structure we can work with throughout the whole embedding. Instead, we have to employ ad-hoc matchings, plus some auxiliary structures, one for each of the connecting vertices. 
 Finally, we have to deal with the leaves adjacent to connecting vertices. These are  more difficult to embed than the other small trees, because an embedded  vertex might only see two thirds of the graph, and there is no way to reach the remaining third of the graph in only one step. For this reason, we have to come up with a delicate strategy on where we place the connecting vertices, in order to ensure that at the very end of the embedding process we will be in a position to embed all these leaves at once with a Hall-type argument.

\section{The proof of Lemma \ref{realhighleafdegreecase} and Theorem~\ref{chighleafdegreecase}}\label{lemmas}

%

The lemma behind our two results from the introduction (Lemma \ref{realhighleafdegreecase} and Theorem~\ref{chighleafdegreecase}) is the following.
 
\begin{lemma}
\label{highleafdegreecase}
For every $\delta>0$, there is an $m_\delta$ such that for any $m \ge m_\delta$ and~$\alpha$ with  $\delta\le\alpha \le 1$ the following holds. \\ Let $G$ be an $(m+1)$-vertex graph of minimum degree at least $\lfloor \frac{2m}{ 3}  \rfloor$, and let  $w\in V(G)$. Let $T$ be a tree with at most $(1-\alpha) m$ edges,  let $t\in V(T)$, and assume that no vertex of $T$ is  adjacent 
to more than $\alpha m$ leaves. Then one can embed~$T$ in $G$, mapping  $t$ to $w$.
\end{lemma}

Let us see how Lemma~\ref{highleafdegreecase} implies the results from the introduction. Here is the proof of the lemma that we will  need in the companion paper~\cite{brucemaya2}.

\begin{proof}[Proof of Lemma~\ref{realhighleafdegreecase}]
Let $m_\delta$ be given by Lemma \ref{highleafdegreecase} for input $\delta$. Given $G$ and $T$ as in Lemma~\ref{realhighleafdegreecase}, 
we let $t$ be a vertex of $T$ having the maximum number of leaf neighbours. Let $L$ be the set of its leaf neighbours and set $\alpha:=\frac{|L|}{m}$. By assumption, $\delta\le\alpha\leq 1$, so we may apply Lemma \ref{highleafdegreecase} to obtain an embedding of $T-L$ in $G$ with $t$ embedded in the universal vertex of $G$. We can then  embed the 
vertices of $L$ into the remaining vertices of~$G$. 
\end{proof}

Here is the proof of the approximate result. 

\begin{proof}[Proof of Theorem~\ref{chighleafdegreecase}]
Let $m_\delta$ be the maximum of the numbers $m_\delta$ given by Lemma \ref{highleafdegreecase} and  by Lemma \ref{realhighleafdegreecase} for input $\delta$. 
Given $G$ and~$T$ as in the theorem, consider a vertex of $T$ with the maximum number of leaf neighbours, say these are $\beta m$ leaf neighbours. If $\beta\le\delta$, we are done by Lemma \ref{highleafdegreecase}. If $\beta>\delta$, we are done by  Lemma \ref{realhighleafdegreecase}. 
\end{proof}

\section{A sketch of the proof of Lemma \ref{highleafdegreecase}}\label{ooover}

The purpose of this section is to give some more detailed insight into the proof of Lemma \ref{highleafdegreecase}, going a little more below the surface than in the Introduction. We remark that for the understanding of the rest of the paper, it is not necessary to read this section (but we hope it will be helpful).

The number $m_\delta$ is chosen in dependence of the output of the regularity lemma for some constant depending on $\delta$. 
Now, given  the approximation constant $\alpha$, the tree $T$ and  the host graph $G$, we prepare each of $T$ and $G$ separately for the embedding. 

Similar as in earlier tree embedding proofs~\cite{AKS95,cite:LKS-cut3,PS07+}, we  cut~$T$ into a set $W$ of  {\it seeds} (connecting vertices), such that $W$ has constant size, and a large set~$\mathcal T$ of very small subtrees.  The trees in $\mathcal T$ are only connected through $W$, and they each have size $<\beta m$, where~$\beta$ is a small constant (smaller than all other constants in this paper). 
Differently from  earlier approaches to tree cutting, we now categorise the small trees contained in~$\mathcal T$.
They fall into three categories: trees consisting of a leaf of~$T$, trees that are smaller than a (huge) constant, and trees that are larger than this constant. We name the categories $L$, $F_1$ and $F_2$. The last category is further subdivided into two sets, $F_2'$ and $F_2\setminus F_2'$, according to whether the small tree is adjacent to one or more of the seeds. Each seed from $W$ may have trees of any (or all) of these categories hanging from it (and there may also be seeds hanging from it). The details of this cut-up of $T$ is explained in Section~\ref{tree-cut}.

Next, in Section~\ref{ordering}, we order and group the seeds obtained from this decomposition.  Our strategy of ordering the seeds takes into account their position in a natural embedding order, but also the number of leaves hanging from them. We will come back to this point at a later stage during this outline, and will then explain the why and how of the ordering.

%


Independently, in Section~\ref{prepas}, we regularize the host graph $G$, with parameter $\eps$, such that $\beta\ll\eps\ll\alpha$. (For an introduction to regularity, see Section~\ref{regularity}.) Furthermore, we partition each of its clusters $C$ arbitrarily into subsets $C_{W}, C_L, C_{F_1}, C_{F_2}, C_{\tilde V}$ of appropriate sizes into which we aim to embed  the different parts of the tree, namely, $W$, $L$, $F_1$, and $F_2$, while the last  subset, $C_{\tilde V}$, is reserved for neighbours of seeds in trees of~$F_2'$. The set of these neighbours will be denoted by $\tilde V$.

We fix a matching $M_{F_2}$ in the reduced graph $R_G$. This matching will be used when we embed the trees from $F_2$. More precisely, we will embed each tree $\bar T\in F_2\cup F'_2$ into $C_{F_2}\cup D_{F_2}$ for a suitable (i.e.~sufficiently unoccupied) edge $CD\in M_{F_2}$, except for the root $r_{\bar T}$ of $\bar T$. The root $r_{\bar T}$ will go to one of the subsets~$C'_{\tilde V}$, for a suitable cluster $C'$ that connects $CD$ with the cluster containing the seed adjacent to $\bar T$. In case $\bar T\in F'_2$, which means that~$\bar T$ contains a second vertex $\tilde v$ from $\tilde V$, we embed  $\tilde v$  into one of the subsets~$C''_{\tilde V}$, for a suitable cluster $C''$. 
Throughout the embedding process, we will keep each of the edges of $M_{F_2}$ as balanced as possible. That is, the sets of used vertices in the corresponding slices $C_{F_2}$ or~$D_{F_2}$ on either side of such an edge never differ by more than $\beta m$.

Since we do not have enough space in the slices $C_{\tilde V}$ for all roots of trees in $F_1$ (because we have no control over the number of trees in $F_1$), we need to proceed differently with the small trees from~$F_1$. For embedding these trees, we use a family of matchings $M_s$, one for each embedded seed $s$. Since these matchings $M_s$ are possibly different for each~$s$, we now will have to keep the set of {\it all} slices~$C_{F_1}$ balanced. This is not easy but possible since the trees from $F_1$ have constant size, and we choose $M_s$ so that it intersects the neighbourhood of the image of $s$ in a nice way. 
In addition to $M_s$, we also employ two auxiliary matchings which  we combine with $M_s$ to obtain a partition of almost all of $V(R_G)$ with short paths. We call these structures {\em good path partitions} and, together with the matchings $M_s$, they will be defined and proved to exist in Subsection~\ref{matchi}.

 
 The actual embedding of the tree will be performed as follows.
In Sections~\ref{emb:seeds} and \ref{emb:trees}, we go through the seeds in a connected way, and embed each seed~$s$ together with all the trees from $F_1\cup F_2$ hanging at $s$ in the corresponding slices in the way we discussed above. We leave out any leaves from~$L$, as we will deal with them in the final phase of the embedding. 
Our precautions from above ensure that we can embed all of $T-L$ without a problem, never running out of space.
However, if we do not take care where exactly we embed the seeds, we may run into problems in the final phase when we need to embed the leaves. For instance, we need to avoid embedding  all parents of $L$  into vertices having the same neighbourhood in $G$, as then, the leaves might not fit.
 
 For this reason, we take some extra care when choosing the target clusters and the  images for the seeds (this happens in Section~\ref{ordering}). As already shortly mentioned above, we
  order the seeds into a system of  groups according to the number of leaves hanging from them, and also according to the order the seeds appear in our planned embedding order. Moreover, each seed $s$ will be assigned a {\it relevant} set $X_s$ of seeds that come before it. 
  In the actual embedding, in Subsection~\ref{emb:seeds},  we choose the image $\varphi(s)$ of a given seed $s$ in a way that $\varphi(s)$ has many neighbours outside the union of the neighbourhoods of  $\varphi(X_s)$. (We remark that it is crucial here that no vertex of $T$ is  adjacent 
to more than $\alpha m$ leaves.) This precaution will ensure that for each subset of seeds, their images have enough neighbours in $Z:=\bigcup C_L$. We will then be able to embed all the leaves in $L$ at once by using Hall's theorem. The whole procedure will be explained in detail in Subsection~\ref{emb:leaves}.

  The last section of this paper, Section~\ref{sec:last}, is devoted to the proof of Lemma~\ref{maya1}, which will need Lemma~\ref{maya1} in our companion paper~\cite{brucemaya2}.  The lemma deals with a similar situation as the one treated in Lemma \ref{highleafdegreecase}, the main difference being that now, a small part of the tree is already embedded (and thus possibly blocking valuable neighbourhoods), but, on the positive side, throughout~\cite{brucemaya2}, we will be able to assume that no seed is adjacent to many leaves, and so we can assume this as well in  Lemma~\ref{maya1}.

\section{Preliminaries}
\subsection{An edge-double-counting lemma}

We will need the following easy lemma.
\begin{lemma}\label{easyA}
Let $G$ be a graph on $n$ vertices, let $0<\psi<\frac 1{3}$, and let $S\subseteq V(G)$. If each vertex in $S$ has degree at least $(\frac 23 - \psi)n$, then there are at least $(\frac 13+\frac{\sqrt\psi}{10})n$ vertices in $G$ that each have at least $(\frac{1}2-\sqrt\psi)|S|$ neighbours in $S$.
\end{lemma}
\begin{proof}
Let $A\subseteq V(G)$ denote the set of all vertices having at least $(\frac{1}2-\sqrt\psi)|S|$ neighbours in $S$. Writing $e(S, V(G))$ for the number of all edges touching~$S$, where edges inside $S$ are counted twice, we calculate that
\begin{align*}
(\frac 23 - \psi)n\cdot |S| & \leq e(S, V(G))\\
& \leq |V(G)\setminus A|\cdot (\frac{1}2-\sqrt\psi)|S| \ +\  |A|\cdot |S|\\
& \leq n \cdot (\frac{1}2-\sqrt\psi)|S| \ +\   |A|\cdot (\frac{1}2+\sqrt\psi)|S|,
\end{align*}
and conclude that $$|A|\ \geq \ \frac{\frac 16+\frac{\sqrt\psi}2}{\frac{1}2+\sqrt\psi}\cdot n\ \geq \ (\frac 13+\frac{\sqrt\psi}{10})\cdot n,$$ as desired.
\end{proof}

\subsection{Matchings and good path partitions}\label{matchi}

The purpose of this subsection is to find some matchings in a graph $H$ (which will later be the reduced graph $R_G$ of our host graph $G$), and combinations of some of these matchings to covers of $H$ with short paths. These structures  will be used for the embedding of $T$ in the proof of Lemma~\ref{highleafdegreecase}, more specifically  in Subsection~\ref{emb:trees}. 
The important result of this section is Lemma~\ref{lem:match_paths}. 

We need a quick definition before we start.
For any graph $H$, and any $N\subseteq V(H)$, an {\em $N$-good}  matching is one whose  edges each have at most one vertex outside $N$. 
 

\begin{lemma}\label{lem:match}
Let $0<\xi<\frac 1{20}$ and let $H$ be a $p$-vertex graph of minimum degree at least $(\frac 23-\xi)p$. Let $N\subseteq V(H)$ be such that $|N|=\lceil  (\frac 23 -2\xi)p\rceil$. Then $H-N$ contains a set $Y$ of size at most $3\xi p +1$ such that $H-Y$ has an $N$-good  perfect matching.
\end{lemma}

 \begin{proof}
First of all, note that we can greedily match all but a set $X$ of at most $3\xi p$ vertices from $V(H)\setminus N$ to $N$, simply because of the condition on the minimum degree. If necessary, add one more vertex to $X$, in order to have that $H-X$ is even. So $|X|\le 3\xi p+1$. Now, take any maximal $N$-good matching $M$ in the graph~$H$ that covers all vertices of~$V(H)\setminus (N\cup X)$. We would like to see that~$M$ covers all  of $H-X$, so for contradiction assume that $N\setminus V(M)$ contains at least two vertices. 

By the maximality of $M$, we know that $N\setminus V(M)$ is an independent set, and no two vertices $D, D'\in N\setminus V(M)$ can be adjacent to different endpoints of an edge in $M$. So, for each edge $EF\in M$, we know that either one of the endvertices, say $E$, sees no vertex in $N\setminus V(M)$, or  $E$ and $F$ each see only one vertex in $N\setminus V(M)$ (and that is the same vertex). Therefore, at least one of the vertices in $N\setminus V(M)$ sees at most half of the vertices in $V(M)$, and thus less than half of the vertices in $H-X$, a contradiction to our condition on the minimum degree.
 \end{proof}

Before we state the second lemma of this section, we need another definition. 
An {\it $N$-out-good path partition} of a graph $H$, with $N\subseteq V(H)$, is a set $\mathcal P$ of disjoint paths, together covering all the vertices of $H$, such that for each $P\in \mathcal P$ one of the following holds:
\begin{itemize}
\item $P=AB$, with $A, B\in N$;
\item $P=ABCD$, with $B,C\in N$; or
\item $P=ABCDEF$, with  $B,C, D, E\in N$.
\end{itemize}
(Note that if $P$ has four vertices, then there is no restriction on the whereabouts of $A$ and $D$, and similar for six-vertex paths $P$.)

An {\it $N$-in-good path partition} of a graph $H$, with $N\subseteq V(H)$, is a set $\mathcal P$ of disjoint paths, together covering all the vertices of $H$, such that for each $P\in \mathcal P$ one of the following holds:
\begin{itemize}
\item $P=AB$, with $A, B\in N$;
\item $P=ABCD$, with $A,D\in N$; or
\item $P=ABCDEF$, with  $A, C, D, F\in N$.
\end{itemize}

Now we are ready to state the main result of this section. Note that the first item is a direct consequence of the previous lemma, and all structures exist independently of each other.

\begin{lemma}\label{lem:match_paths}
Let $0<\xi<\frac 1{20}$, and let $H$ be a $p$-vertex graph of minimum degree at least $(\frac 23-\xi)p$. Let $N\subseteq V(H)$ be any set with $|N|=\lceil (\frac 23 -2\xi)p\rceil$. Then $H$ contains a set $X$ of at most $\lfloor 23\xi p\rfloor +3$ vertices, such that 
\begin{itemize}
\item $H-X$ has an $(N\setminus X)$-good perfect matching; 
\item $H-X$ has an $(N\setminus X)$-in-good path partition; and
\item $H-X$ has an $(N\setminus X)$-out-good path partition.
\end{itemize}
\end{lemma}

 \begin{proof}
 Lemma~\ref{lem:match} provides us with a set $Y$ and an
 $N$-good  perfect matching~$M$ of $H-Y$. Note that $M$ would be as desired for the first item, if $X$ was chosen as $Y$. 
 
 Now, set $Q:=V(M)\setminus N$, and let $R$  be the set of all vertices from $H$ that are matched by $M$ to a vertex from $Q$. Since $M$ is $N$-good, we know that $R\subseteq N$.  Note that 
 \begin{equation}\label{specials}|N\setminus R|\ge |N| - |Q| \ge |N|- (p-|N|)=2|N|-p\ge (\frac 13-4\xi )p.
 \end{equation}

We take a maximal matching  $\tilde M^{Q}$ inside $H[Q]$. Let $\tilde{Q}$ be the set of vertices in $Q$ not covered by $\tilde M^{Q}$. We augment  $\tilde M^{Q}$ to a  matching $M^{Q}$ by  matching as many vertices of $\tilde Q$ as possible to vertices of $N\setminus R$. Let $Z$ denote the set of remaining vertices of $\tilde Q$. Note that each vertex $v\in Z$ sees at most one of the endvertices of any edge of $M^{Q}\setminus \tilde M^{Q}$, as $\tilde Q$ is independent, and moreover, $v$ does not see $(N\setminus R)\setminus V(M^{Q})$. Therefore, for each $v\in Z$, we have
\begin{align*}(\frac 23-\xi)p\le |N(v)|\le |\tilde Q\setminus Z|+ |V(\tilde M^{Q})|+|Y|+|R| & =|Q\setminus Z|+|Y|+|R|\\ & =p-|N\setminus R|- |Z|,\end{align*}
and hence, by~\eqref{specials}, we have $$|Z|\le (\frac 13+\xi )p- |N\setminus R| \le 5\xi p.$$
 Let  $Z'\subseteq R$ be the set of all vertices matched to vertices of $Z$ by $M$.

We define a second auxiliary matching  $M^R$ in a very similar way. Namely, $M^R$ consists of edges with both ends in $R\setminus Z'$, and a matching of almost all the remaining vertices of $R\setminus Z'$ to  vertices of $N\setminus R$.
We now take some more care: For each pair of vertices $v, v'\in R\setminus Z'$ whose partners $w,w'$ in $M$ form an edge in $M^Q$, we will match either both of $v,v'$ or none (in each of the two matching steps). Letting $Z''$ denote the set of all unmatched vertices of $R\setminus Z'$, we can calculate similarly
 as above that for each vertex $v$ from $Z''$, either $v$ or the paired vertex $v'$ has degree at most $|R\setminus Z''|+ |Q|+|Y|+1$. This leads to the bound  $|Z''|\le 5\xi p+1$. 
 
 Let $Z'''\subseteq Q$ denote the set of all  partners of vertices of $Z''$ in $M$. Set $X:=Y\cup  Z\cup  Z'\cup  Z''\cup Z'''$. Note that $$|X|\le 3\xi p +1+ 2\cdot 5\xi p +2\cdot (5\xi p +1) =23\xi p+3.$$

Now, discard any edge that touches $X$ from $M$, $M^Q$ and $M^R$, thus obtaining matchings $\bar M$, $\bar M^Q$ and $\bar M^R$. Then  the union $U$ of the edges in $\bar M\cup \bar M^Q$ gives an $(N\setminus X)$-in-good path partition of $H-X$. Indeed, 
consider a component $P$ of $U$. 

First assume $P$ contains an edge $BC$ from $\bar M^Q$. If $B, C\in Q$, then (by the choice of $\bar M^Q$), we know that $\bar M$ contains edges $AB$ and $CD$, with $A,D\in R$. The only other case (modulo renaming $B$ and $C$) is that $B\in Q$, and $C\in N\setminus R$. Then again, (by definition of $\bar M^Q$), we know that $\bar M$ contains edges $AB$ and $CD$, with $A\in R$ and $D\in N\setminus R$.
If $D$ is not matched in $\bar M^Q$, then $ABCD$ is as desired for the path partition. If $D$ is matched in $\bar M^Q$, say to a vertex $E\in Q\setminus X$, then $ABCDEF$ is as desired for the path partition, where $F$ is the partner of $E$ in $\bar M$. 

Now assume $P=AB$ is an edge from $\bar M$. If $A,B\in N$, then $P$ is as desired, so assume $B\in Q$. Then $B$ is matched to some vertex $C\in Q$ in $M^Q$, but $BC\notin \bar M^Q$. Thus $C\in Z'''$. However, because of the way we chose $M^R$ and $Z''$, we know that $A$ is matched in $M^R$, which means that also the partner of $C$ in $M$ is matched in $M^R$, and thus $C\notin Z'''$, a contradiction.

Similarly, the union of the edges in $\bar M\cup \bar M^R$ gives the desired $N$-out-good path partition. 
 Finally, $\bar M$ is an $(N\setminus X)$-good matching of $H-X$.
\end{proof}

\subsection{Regularity}\label{regularity}

We need to quickly discuss Szemer\'edi's regularity lemma and a couple of other preliminaries regarding regularity. Readers familiar with this topic are invited to skip this subsection.

The {\it density} of a pair $(A,B)$ of disjoint subsets $A,B\subseteq V(G)$ is $d(A,B)=\frac{|E(A,B)|}{|A|\cdot |B|}$.
A pair $(A,B)$ of disjoint subsets $A,B\subseteq V(G)$ is called \emph{$\eps$-regular} if
$$ |d(A,B) - d(A',B')| < \eps$$ for all  $A' \subseteq A,~B'\subseteq B$ with $|A'|\geq \eps |A|,~|B'| \geq \eps|B|$.
In many ways, regular pairs behave  like  random bipartite graphs with the same edge density. 

If $(A,B)$ is an $\varepsilon$-regular pair, then
we call a subset $A'$ of $A$ {\it $\varepsilon$-significant} (or simply significant, if $\varepsilon$ is clear from the context) if $|A'|\geq \eps |A|$.
We call a vertex from $A$ {\em $\varepsilon$-typical} (or simply typical, if $\varepsilon$ is clear from the context) with respect to a set~$B'\subseteq B$ if it has degree at least $(1- \eps)d(A,B)|B'|$ to~$B'$. 

The following well known and easy-to-prove facts (see for instance~\cite{regu2}) state that in a regular pair almost every vertex is typical to any given significant set, and also that regularity is inherited by subpairs. 
More precisely, if $(A,B)$ is an $\varepsilon$-regular pair with density $d$, then 
\begin{itemize}
 	\item[(R1)] for any $\varepsilon$-significant $B'\subset B$, all but at most $\varepsilon|A|$ vertices from $A$ are $\varepsilon$-typical to $B'$; and
 	\item[(R2)] \label{fact:1,2} for each $\delta\ge 0$, and for any subsets $A'\subseteq A$, $B'\subseteq B$, with $|A'|\ge\delta|A|$ and $|B'|\ge\delta|B|$, the pair $(A', B')$ is $\frac{2\varepsilon}{\delta}$-regular with density between $d-\varepsilon$ and $d+\varepsilon$.
\end{itemize}

Szemer\'edi's regularity lemma states that every large enough graph has a partition of its vertex set into one `trash' set of bounded size, and
 a bounded number of sets of equal sizes, 
such that almost all pairs of these sets are $\eps$-regular.

\begin{lemma}[Szemer\'edi's regularity lemma]
\label{RL}
For each $\eps > 0$ and  $M_0\in \mathbb{N}$ there are $M_1,n_0 \in \mathbb{N}$ such that for all $n \geq n_0$, 
every $n$-vertex 
graph $G$ has a partition $V_0\cup V_1 \cup \ldots \cup V_p$
of $V(G)$ into $p+1$  partition classes (or \emph{clusters}) such that
\begin{enumerate}[\rm (a)]
\item \label{itm:regularity-a} $M_0 \leq p \leq M_1;$
\item $|V_1| = |V_2| = \ldots = |V_p|$ and $|V_0| < \eps n;$
\item apart from at most  $\eps \binom{p}{2}$ exceptional pairs, the pairs $(V_i, V_j)$ are $\eps$-regular, for $i,j>0$ with $i\neq j$.
\end{enumerate}
\end{lemma}

As usual, we define the {\it reduced graph} $R_G$ corresponding to this decomposition of $G$ as follows. 
The vertices of $R_G$ are all  clusters $V_i$ ($i=1,\ldots,p$), and $R_G$ has a edge between $V_i$ and $V_j$ if the pair $(V_i,V_j)$ is $\eps$-regular, and has density at least $10\sqrt\eps$. 
By standard calculations (see for instance~\cite{regu2}), and assuming we take $M_0\geq \lceil \frac 1\eps\rceil$, it follows that
\begin{equation}
 \text{$\delta_w(R_G)\geq (1-12\sqrt{\epsilon})\cdot \frac p{|V(G)|} \cdot\delta (G)$,}\label{R_Gmindeg}
 \end{equation}
 where $\delta_w(R_G)$ is the weighted minimum degree. (That is, the densities of the pairs of clusters provide weights on the edges of $R_G$, and the weighted degree of a vertex is the sum of the corresponding edge-weights. The weighted minimum degree is the minimum of these weighted degrees. Observe that $\delta_w(R_G)\leq\delta (R_G)$ since weights do not exceed $1$.)

Almost all vertices
of any cluster $C\in V(R_G)$ are {typical} to almost all significant sets,
in the following sense.
If $\mathcal Y$ is a set of significant subsets of clusters in $V(R_G)$, then
\begin{align}\label{cyro}
&\text{all but at most
$\sqrt{\varepsilon}|C|$ vertices $v\in C$ are typical with respect to}\notag \\ &\text{all but at most  $\sqrt{\varepsilon}|\mathcal Y|$ sets in $\mathcal Y$.}
\end{align}

To see this well-known observation, assume that the set $C'\subseteq C$ of vertices not
satisfying~\eqref{cyro} is larger than $\sqrt{\varepsilon}|C|$. Then
\begin{align*}
\sum_{Y\in \mathcal Y}|\{v\in C:\; v\text{ is not typical to }Y\}| \geq &\ \sum_{v\in C'}|\{Y\in \mathcal Y:\; v\text{ is
not typical to }Y\}|  \\
\geq  & \ |C'|\sqrt{\varepsilon}|\mathcal Y|\\  > & \ \varepsilon|C|\cdot |\mathcal Y|.
\end{align*}
Thus there is a $Y\in \mathcal Y$ such that more than $\varepsilon |C|$
vertices in $C$ are not typical to~$Y$, a contradiction.
\

Regularity will help us when embedding small trees into a pair of adjacent clusters of  $R_G$. 

\begin{lemma}\label{treeinpair}
Let $CD$ be an edge of $R_G$, and let $U\subseteq G$ with $|C\setminus U|, |D\setminus U|\geq \sqrt\eps  |C|$. Let $\bar T$ be a tree of size $\leq \varepsilon |C|$ with root $r_{\bar T}$. \\  Then $\bar T$ can be embedded into $G$, with $\bar T-r_{\bar T}$ going to  $(C\cup D)\setminus U$, and with $r_{\bar T}$ going to any prescribed set of $\geq 2\eps  |C|$ vertices of $C$, or to any prescribed set of $\geq 2\eps |C|$ vertices of $C'$, where $C'$ is any other cluster of $R_G$ that is adjacent to $D$.
\end{lemma}

\begin{proof}
We construct the embedding  $\bar T$ levelwise, starting with the root, which is embedded into a typical vertex of $(C\cup D)\setminus U$. At each step~$i$ we ensure that all vertices of level $i$ are embedded into vertices of $C\setminus U$ (or~$D\setminus U$) that are typical with respect to the unoccupied vertices of $D\setminus U$ (or $C\setminus U$). This is possible, because at each step $i$, and for each vertex $v$ of level $i$, by (R1), the degree of a typical vertex into the unoccupied vertices on the other side is at least $4\varepsilon |C|$, and there are at most $\varepsilon |C|$ nontypical vertices and at most $|\bar T|\leq \varepsilon |C|$ already occupied vertices.
\end{proof}

\section{Preparing the tree}\label{sec:tree}
\subsection{Cutting a tree}\label{tree-cut}

In this section, we will show how any tree $T$ can be cut up into small subtrees and few connecting vertices. The ideas is that later, we can use regular pairs to embed many tiny trees.

We will make use of a procedure which in a very similar shape has already appeared in~\cite{AKS95,cite:LKS-cut3,PS07+}, although there, no distinctions between the trees from $L$, $F_1$, $F_2$, were made. The resulting cut-up is given in the following statement.

\begin{lemma}\label{cutT}
For any $m\in \mathbb N$,  for any tree $T$ on $m+1$ vertices, for any $r\in V(T)$, and for any $\beta>0$,
there is a set $W\subseteq V(T)$, and a partition $\mathcal T=L\cup F_1\cup F_2$ of the family $\mathcal T$ of components of~$T-W$, distinguishing a subset $F_2'\subseteq F_2$, such that 
\begin{enumerate}[(a)]
\item $r\in W$;\label{root}
\item $|W|\leq \frac{2}{\beta^2}$\label{few}; 
\item if $\beta m>1$ then each $w\in W$ has a child in $T$;\label{noleafseed}
\item $|V(\bar T)|=1$ for every tree $\bar T\in L$;\label{leavesofthetree}
\item $1<|V(\bar T)|\leq \frac 1\beta$ for every tree $\bar T\in F_1$;\label{tiny}
\item $\frac 1\beta<|V(\bar T)|\leq \beta m$ for every tree $\bar T\in F_2$;\label{small}
\item  \label{oneneighbour} each $\bar T\in L\cup F_1\cup (F_2\setminus F_2')$ has exactly one neighbour in $W$; 
\item  \label{atmosttwo} each $\bar T\in F_2'$ has exactly two neighbours in $W$; and
\item \label{eqXX} $|\tilde V|<2\beta m$, where $\tilde V$ is the set of all neighbours of vertices of $W$ in $\bigcup F_2'$.
\end{enumerate}
The vertices in $W$ will be called the {\em seeds} of $T$.
\end{lemma}

\begin{proof}
In a sequence of at most $\frac 1\beta$ steps $i$, we define vertices $w_i$ and trees~$T_i$ as follows. Set $T_0:=T$. Now, for each $i>0$, let $w_i\in V(T_{i-1})$ be a vertex at maximal distance from $r$ (the root of $T$) such that the components of $T_i-w_i$ that do not contain $r$ each have size at most $\beta m$. Delete $w_i$ and all of these components from $T_{i-1}$ to obtain $T_i$. We stop when we reach $r$, which will be the last vertex $w_i$ to be defined. 

Let $W_0$ be the union of all $w_i$, and let $\mathcal T_0$ be the set of all components of $T-W_0$. These two sets already fulfill items $(\ref{root})$, $(\ref{few})$  and~\eqref{noleafseed}. (To see $(\ref{few})$, note that at each step $i$, we cut off $\beta m$ vertices. Hence we actually have that $|W_0|\leq \frac 1\beta$.) 

In order to obtain sets $W$, $\mathcal T$ that also fulfill items $(\ref{oneneighbour})$ and~\eqref{atmosttwo}, we add some vertices to $W_0$ as follows. For each $\bar T\in \bigcup\mathcal T_0$ that has $\ell >2$ neighbours $v_1, v_2,\ldots, v_\ell$ in $W_0$, we add to $W_0$ a set of at most $\ell -1$ vertices $w'_j$ from $V(\bar T)$ that separate all $v_i$'s from each other. Note that these are at most $\frac 1\beta$ vertices in total (counting over all affected $\bar T$), since each of the newly added vertices $w'_j$ can be associated to one of the `old' vertices $v_j$ from $W_0$ such that  $w'_j$ lies between~$v_j$ and $r$. 
So, letting $\mathcal T_1$ be the family of the trees in $T-W_1$, the new sets $W_1$, $\mathcal T_1$ still fulfill $(\ref{root})$, $(\ref{few})$  and~\eqref{noleafseed} (actually, we have that $|W_1|\leq \frac2\beta$). They  furthermore have the property that each of the trees in $\bigcup\mathcal T_1$ has at most two neighbours in $W_1$.

We modify our sets once more to ensure that only the large trees  can have two seed neighbours. We proceed as follows. For each $\bar T\in \bigcup\mathcal T_1$ that has at most $\frac 1\beta$ vertices and is adjacent to two seeds $w_1, w_2\in W_1$, we add to $W_1$ all vertices in $V(\bar T)$.
 In total, these are at most $\frac{1}{\beta}\cdot |W_1|\leq \frac2{\beta^2}$ vertices.  Call the new set of seeds $W$.

Defining $\mathcal T$ as the family of the trees in $T-W$, and 
adequately dividing~$\mathcal T$  into three sets $L, F_1, F_2$, and letting $F_2'$ be the appropriate subset of $F_2$, we  obtain sets  that fulfill all properties of the claim (where~\eqref{eqXX}  follows directly from~\eqref{small}, \eqref{atmosttwo} and the observation that $|\tilde V|\leq 2|F_2'|<\frac{2|T-r|}{\frac 1\beta}=2\beta m$).
\end{proof}

 \subsection{Ordering the seeds of a tree}\label{ordering}
 
In order to be able to  choose well the clusters  of $V(R_G)$ into which we will embed the seeds  other than $r$ later on, we will  define a convenient ordering of the seeds of a tree $T$ with a cut-up as in Section~\ref{tree-cut}. Together with this ordering we will define a set of relevant seeds $X_s$ for each seed $s$ of the tree, and ensure that the seeds in $X_s$ come before $s$ in the ordering.

 The purpose of this ordering and the definition of the sets $X_s$ is that later, when we embed the trees from $L$, $F_1$ and $F_2$ in $G$, it will turn out that the smaller the trees, the harder they are to embed, with the most difficult ones being the trees from $L$, i.e., the leaves of $T$ adjacent to seeds. An embedded seed has only degree $\frac 23m$ in $G$, of which a large part might already be used, so we have to plan ahead in order to avoid  getting stuck when embedding the leaves. For this reason we have to choose very well into which clusters the seeds go, and the sets $X_s$ will help us with this.

The reader might wish to skip the remainder of this unfortunately rather technical section at a first reading, because everything we do here is only necessary for the embedding of $L$. Even the embedding of $L$ can be followed with only a vague understanding of  the definitions of the present section if the reader takes the `Degrees of the embedded seeds into $Z$' as stated in Subsection~\ref{degpropembseeds} for granted.

\subsubsection{Grouping and ordering}\label{groupandorder}

Let us  start with the ordering. Assume we are given a tree $T$ which has been treated by Lemma~\ref{cutT} for some $\beta >0$, let $W$ denote the set of seeds we obtained. Throughout the rest of this section we will assume that
\begin{equation}\label{number_of_seeds}
\text{$|W|=47\cdot 2^{j^*}$, where $j^*=\lceil \log\frac 2{\beta^2}\rceil$.}
\end{equation}
(This can be assumed by adding some new vertices to the tree $T$, and declaring all of them seeds. We will explicitly discuss why this can be done in Subsection~\ref{prepTemb}.)

We will order the seeds in two different ways, before we get to the third and final order. The first order is determined by the number of leaves hanging from each seed, the second order is determined by the  position of the seeds in the tree $T$, and the third order is a mixture of both. We explain the orderings in detail in the following.

We start by ordering the seeds in a way that the number of leaf children of the seeds is decreasing, and we call this the {\it size order} $\sigma$ on the seeds. Now, we will define ordered sets of seeds, which we will call {\it groups} of seeds. First we will define the {\it large groups}: loosely speaking, the largest of these groups consists of all seeds, then we define two groups consisting of the first and the second half of the seeds respectively, then each of these groups is divided into two new groups, and so on, until the size is down to 47.
More precisely,
 for each $j=0,\ldots ,j^*$, we partition the set of all seeds into $2^{j^*-j}$ consecutive groups of    size $2^j\cdot 47$, under the size order $\sigma$, and we call these the {\it large groups}.  
 Clearly, each large group of  size exceeding 47  is the union of two large groups  half its size. 
 
We break up each group $B$ of size 47 into twelve consecutive groups (consecutive under $\sigma$) of the following sizes:
\begin{equation}\label{seq}
 4, {\bf 4}, 4, 4,  5, 4, {\bf 4}, 4, 4,  5, 4, 1.
 \end{equation}
  We call these the {\it small groups}. (So the small subgroup of size $1$ of $B$ consists of the very last seed of $B$ in the size order $\sigma$.) 
We say the second  and sixth group of size four are of type 1  (they are marked in boldface in~\eqref{seq}). The remaining groups of size four (i.e.~the first, third, fourth, fifth, seventh, eighth and ninth group of size four) will be called type 2.  

It would be difficult to embed the seeds in the size order $\sigma$, as this enumaration might not be suitable for embedding the tree in a connected way. For this reason, we employ a second order~$\tau$, which we call the {\it transversal order}, obtained by performing a preorder tranversal on  $T$, starting with the root~$r$, and then restricting this order to $W$.  (The transversal order  is the actual order the seeds will be embedded in.)

The third order, which we call the {\it rearranged order $\rho$}, is obtained by reordering the order $\sigma$. First, we reorder the seeds in each small group so that each small group is ordered according to $\tau$. Next, for every large group~$B$ of size~$47$, we reorder all its subgroups so that their first seeds form an increasing sequence in the transversal order $\tau$. Finally,  for every large group~$B$ of size $>47$ (in successive steps according to the group size), we reorder the  two subgroups  within  $B$  so that the first subgroup contains the first seed in $B$ under the transversal order $\tau$ (i.e.~we reorder them such that the first seed of $B$ under $\tau$ becomes the first seed of $B$).
This finishes the definition of the rearranged order $\rho$.

We note that $\rho$ maintains the structure given by breaking down the set of seeds into large and small groups. That is, if we partition the set of all seeds into $2^{j^*-j}$ consecutive groups under $\rho$ of  sizes $2^j\cdot 47$, we obtain the same groups as above for $\sigma$. Further, each group~$B$ of size $47$ breaks down into twelve small groups as above, although these are no longer ordered as in the sequence from~\eqref{seq}. (For instance, in $\rho$, the small group of size 1 from~$B$ could become the first group in~$B$, or be at any other position.)

We will embed according to $\tau$ but momentarily  work  with~$\rho$. We write $s<_\rho s'$ to denote that $s$ comes before~$s'$ in  order  $\rho$ (and similar for~$\tau$).

\subsubsection{Sequences}\label{sequencesandlevels}
In this subsection, we will follow the rearranged order $\rho$. 
We define for each large group $B$ two sequences $$(x^B_i)_{i=1\ldots, j+6}\text{ and }(y^B_i)_{i=1\ldots, j+7},$$ where $j$ is such that $|B|=2^j\cdot 47$, and vertices $x^B_i, y^B_i\in B$ are as specified in what follows.

We construct our sequences inductively.
For $j=0$, we have $2^{j^*}$ large groups of size $47$. For each such large  group $B$, we take $x^B_1=y^B_1$ as the first seed of the group. The first seed of the second, third, fourth, fifth and sixth small subgroup of $B$ is chosen as $x^B_2, x^B_3, x^B_4, x^B_5, x^B_6$, respectively.  The first seed of the seventh, eigth, ninth, tenth, eleventh and twelfth small subgroup of $B$ is chosen as $y^B_2, y^B_3, y^B_4, y^B_5, y^B_6, y^B_7$, respectively. (We always work under $\rho$, both when talking about the `first seed of a group' and when talking about the `$i$th subgroup'.)

For $j\geq 1$, we have to deal with all large groups of size $2^j\cdot 47$. For each such group $B$, do the following. By construction $B$ is made up of two large subgroups of size $47\cdot 2^{j-1}$, say these are $B'$ and $B''$ (in this order, under $\rho$). We set $$x^B_i:=y^{B'}_i\text{ for all }1\leq i\leq j+6,$$ and  $$y^B_1:=x^B_1=y^{B'}_1,\text{\  and \ }y^B_i:=y^{B''}_{i-1}\text{ for all }2\leq i\leq j+7.$$ This finishes the definition of the sequences. We remark that we will only use the sequences $(x_i)$ in what follows (the sequences $(y_i)$ were only used to make the definition of $(x_i)$ more convenient).

Observe that for all blocks $B$, and for all $i<j$, we have that $x^B_i<_\tau x^B_j$.

\subsubsection{Relevant seeds}

In order to be later able to choose well the clusters we embed the seeds into (which in turn will enable us to embed the leaves at an even later stage), we need to define, for each seed $s$, a set $X_s$ of relevant seeds for $s$, as follows. 
\begin{definition}[Relevant seeds for $s$]\label{relevantseeds}\ \\
Let $s$ be a seed of $T$, and let $B$ be the small group $s$ belongs to. 
\begin{enumerate}[(a)]
\item If $B$ is a group of four of type 2, and $s$ is the last seed of $B$, then we set  $$X_s:=\{x\ : \ x\text{ is the third seed in }B\text{ (under $\rho$)}\}.$$
\item If $s$ is not the first seed of $B$, and, in case $B$ is a group of four of type~2, $s$ is not its last seed,  then we set $$X_s:=\{x\ : \ x\in B, \ x<_\rho s\}.$$
\item If $s$ is the first seed of $B$, then we set $$X_s:=\{x\ : \ \exists \tilde B, i, i'\text{ such that~$i'< i$, }s=x^{\tilde B}_i\text{ and }x=x^{\tilde B}_{i'}\}.$$ Observe that $\tilde B$, if it exists, is unique, and that if $s$ only appears as a first vertex in  the sequences~$(x^{\tilde B}_i)$, then $X_s=\emptyset$.
\end{enumerate}
\end{definition}

Let us make a quick observation which follows directly from the definition of the order $\rho$, of the sequences $(x_i)$ and of the sets~$X_s$.
\begin{observation}\label{obs1}
Let $s$ be a seed. Then for all $x\in X_s$ it holds that $x<_\tau s$. 
\end{observation}


\section{The proof of Lemma \ref{highleafdegreecase}}\label{sec:high}
\subsection{Preparations}
\subsubsection{Setting up the constants}\label{prepas}

First of all, given $\delta$, we choose 
\begin{equation}\label{alpha_eps_beta}
\eps\leq \frac{\delta^4}{10^{18}},
\end{equation} 
 and apply Lemma~\ref{RL}  (the regularity lemma) with input $\eps^2$ and $M_0:=\frac 1{\epsilon^{2}}$. This yields numbers $M_1$ and $n_0$. 
We then set $$\beta := \frac{\eps}{100 M_1}.$$
Finally, we choose 
\begin{equation}\label{m_delta}
m_\delta:=(n_0+1)\cdot\frac{400 M_0}{\beta^{10}\cdot\eps\cdot\delta}
\end{equation} 
 for the output of Lemma~\ref{highleafdegreecase}. 
So, given the approximation constant $\alpha$, satisfying $1\ge \alpha\ge \delta$, 
we will have that
\begin{equation}\label{beta_eps_lambda_alpha}
0<\frac 1{m_\delta}\ll\beta\ll\eps\ll\delta\le\alpha,
\end{equation} 
with the explicit dependencies given above.
Now, given $m\geq m_\delta$, and given an $(m+1)$-vertex graph $G$ of minimum degree at least $\lfloor \frac{2m}{ 3}  \rfloor$ , and  a tree~$T$ with at most $m-\alpha m$ edges, rooted at $r$, we will prepare both $T$ and $G$ for the embedding.

\subsubsection{Preparing  $T$ for the embedding} \label{prepTemb}

We apply Lemma~\ref{cutT} to obtain a partition of $T$ into a set $W$ of seeds and a set~$\mathcal T$ of small trees. The small trees divide into $L$, $F_1$ and $F_2$, with two-seeded trees $F'_2\subseteq F_2$, and the lemma also gives us a set $\tilde V$.
Set $$f_1:=\sum_{\bar T\in F_1}|V(\bar T)|\ \text{ and }\ f_2:=\sum_{\bar T\in F_2}|V(\bar T)|.$$ 

Next, add a set $W'$ of vertices to $T$, each adjacent to  $r$, 
such that, setting $\tilde W:=W\cup W'$, we have  $$|\tilde W|=47\cdot 2^{j^*}$$ for \begin{equation}\label{defj^*}
j^*:=\lceil \log\frac 2{\beta^2}\rceil.
\end{equation} The only reason for this is that we plan to apply the grouping and ordering of seeds from Subsection~\ref{ordering}, that is, we would like to see~\eqref{number_of_seeds} fulfilled. 
We are going to embed $T\cup W'$ instead of $T$. Since the number of vertices in $W'$ is a constant, space is not a problem. Indeed, 
clearly, 
\begin{equation}\label{eqX}
|\tilde W|+|L|+f_1+f_2\ =\ |V(T)|+|\tilde W\setminus W|\ \leq \ m-\alpha m+\frac{200}{\beta^2}.
\end{equation}

\subsubsection{Preparing  $G$ for the embedding} 

As a preparation of $G$ for the embedding, we  take an $(\eps^2/2)$-regular partition of~$G$ as given by Lemma~\ref{RL} (the regularity lemma), into $p$ clusters, for some~$p$ with  $M_0<p<M_1$.
Consider the reduced graph $R_G$ of $G$ with respect to this partition, as defined below Lemma~\ref{RL}.

Because of the minimum degree of  $G$ and by~\eqref{R_Gmindeg}, we have that
\begin{equation}\label{minRG}
\delta_w (R_G)\geq (\frac 23-13\eps)p.
\end{equation}

Let us now partition the clusters of $R_G$ further. We will divide each clusters into several slices, into which we plan to embed the distinct parts of the tree $T$ which we identified  above.

First of all, we 
choose a
 set $Z$ of vertices into which we plan to embed $L$. More precisely, we arbitrarily choose a set  $Z\subseteq V(G)$ of size 
\begin{equation}\label{sizeofZforleaves}
|Z|=|L|+\lceil (\alpha-\frac{\alpha^4}{10^6})m\rceil,
\end{equation} 
choosing the same number of vertices in each part of the partition (plus/minus one vertex).
Now, we will split up the remainder $C\setminus Z$ of each cluster $C\in V(R_G)$ 
arbitrarily 
into four sets $C_{\tilde V}$, $C_{W}$, $C_{F_1}$, $C_{F_2}$, and a leftover set $C\setminus (Z\cup C_{\tilde V}\cup C_{W}\cup C_{F_1}\cup C_{F_2})$ which will not be used. The sets are chosen having the following sizes:
\begin{equation}\label{slicesVW}
 \text{$|C_{\tilde V}|=|C_{W}|=\lceil \frac{\frac{\alpha^4}5 m}{p} \rceil$;}
\end{equation}
\begin{equation}\label{slicesF1}
 \text{ $|C_{F_1}|=\lceil \frac{f_1+\frac{\alpha^4}5 m}p \rceil$;}
\end{equation}
 and
\begin{equation}\label{slicesF2}
 \text{ $|C_{F_2}|=\lceil \frac{f_2+\frac{\alpha^4}5  m}p \rceil$.}
\end{equation}

This is possible because of~\eqref{beta_eps_lambda_alpha} and~\eqref{eqX}.
As we mentioned above, the idea behind this slicing up is that we are planning to put each part~$X$ of the tree ($X\in\{W$, $\tilde V$, $F_1$, $F_2$, $L\}$) into the parts $C_X$ of the clusters of $R_G$, or into $Z$, respectively. We reserve a bit more than is actually needed for the embedding, in order to always be able to choose well-behaved (typical) vertices, and also in order to account for slightly unbalanced use of the regular pairs when embedding the trees from~$\mathcal T$. Since the sets $C_X$ are large enough, regularity properties will be preserved between these sets (cf.~Section~\ref{regularity}).

Let us remark that it is not really necessary to slice the clusters $C$ up as much as we do: the vertices destined to go into slices $C_{\tilde V}$ and $C_{W}$ are actually so few that they could go to any other slice without a problem. But we think the exposition might be clearer if everything is well-controlled.


Finally, we fix a perfect matching $M_{F_2}$ of $R_G$ which exists because of~\eqref{minRG}. This matching will be used for embedding the trees from  $F_2$.

\subsubsection{The plan}

For convenience, for each seed $s\in \tilde W$, let $\mathcal T_s$ denote the set of all trees from $\mathcal T\setminus L$ that  hanging from $s$.
We are going to traverse the seeds in the transversal order $\tau$, placing each seed~$s$  into a suitable cluster $S(s)$ (we will determine this cluster right before embedding $s$ into it). 
We then embed $\bigcup\mathcal T_s$  before embedding any other seed.  After having embedded all seeds $s\in \tilde W$ and all corresponding trees from $\bigcup\mathcal T_s$, we embed all of $L$ in one step at the very end of the embedding process. So, if the seeds are ordered as $s_1, s_2, s_3, \ldots, s_{|\tilde W|}$ in $\tau$, then we embed in the order $$s_1, \ \bigcup\mathcal T_{s_1}, \ s_2, \ \bigcup\mathcal T_{s_2}, \ s_3, \ \bigcup\mathcal T_{s_3}, \ \ldots s_{|\tilde W|}, \ \bigcup\mathcal T_{s_{|\tilde W|}}, \ L,$$
and at every point in time, the embedded parts of the tree will form a connected set in $T$.

Each of the three different embedding procedures  will be described in detail in one of the following subsections, namely, in Subsection~\ref{emb:seeds} (embedding a seed~$s$), in Subsection~\ref{emb:trees} (embedding $\bigcup \mathcal T_s$) and in Subsection~\ref{emb:leaves} (embedding~$L$).

\subsection{Embedding the seeds}\label{emb:seeds}
\subsubsection{Preliminaries}
Assume we are about to embed some seed~$s$.  Denote by $U$ the set of vertices that, up to this point, have been used for embedding seeds and small trees. So $U\cap Z=\emptyset$ (we will ensure that this will always remain so), and every cluster $C\in V(R_G)$ divides into six sets: $C\cap U$, $C\cap Z$, $C_{W}\setminus U$, $C_{\tilde V}\setminus U$, $C_{F_1}\setminus U$, and $C_{F_2}\setminus U$.

Apart from $U$, it will be useful to have a set $U'\subseteq\bigcup_{C\in V(R_G)}(C_{F_1}\setminus U)$ of vertices for which at some point we decided that they will never be used for the embedding. The main purpose of this set $U'$ is that 
 after embedding certain trees from $\mathcal T_{s}\cap F_1$ for some seed $s$, we can just make all sets $C_{F_1}$ of clusters~$C$ equally `occupied' by discarding some of the vertices of the emptier sets $C_{F_1}$ by putting them into $U'$. This will be the only time we add vertices to $U'$.
 We will make sure that for each seed $s$ the number $u'_s$  of vertices we add to~$U'$ while, or directly after, embedding $\mathcal T_s$ is bounded by
  \begin{equation}\label{U'small}
  u'_s\leq \frac 3{\beta^{10}}+800\eps \cdot |\mathcal T_s\cap F_1|.
  \end{equation}
 Since there at most $\frac 2{\beta}$ (original) seeds in the tree, this means that 
 at any time,
 $$|U'|\le \frac 6{\beta^{11}}+800\eps \cdot \sum_{s\in \tilde W}|\mathcal T_s\cap F_1|\le 801\eps m.$$
 In other words, the set $U'$ will always stay so small that we can ignore it while embedding.

Throughout the embedding, we will ensure that for each parent  $u$ of a seed (the parent $u$ might be a seed, or a vertex from $\tilde V$) the following holds. If $u$ was embedded in  vertex $\varphi (u)$, then we have that 
\begin{equation}\label{goodparent}
\text{$\varphi (u)$ is typical to slice $C_{W}$ for all but at most $\eps p$ clusters $C$ of $R_G$.}
\end{equation}

Note that by Observation~\ref{obs1},  by the time we reach a seed $s$, the `relevant' seeds in $X_s$ have already been embedded into a set $\varphi (X_s)$. Let $N_Z(X_s)$ denote the set of all neighbours of vertices from $X_s$ in $Z$, i.e.~$N_Z(X_s):=N(X_s)\cap Z$. Let $\mathcal N_s$ be the set of the corresponding subsets of the clusters of $R_G$ (i.e.,~$\bigcup\mathcal N_s=N_Z(X_s)$).

\subsubsection{Finding the target cluster $S(s)$ for $s$}\label{sec:target}

Before actually choosing the vertex $\varphi (s)$ we will embed $s$ into, we will determine the target cluster $S(s)$ for a seed~$s$. 

Observe that by Lemma~\ref{easyA}, with $\psi:=13\eps$, we know that at least $(\frac 13 + \eps^{\frac 13})m$ of the vertices of $G$ see a $(\frac 12-\eps^{\frac 13})$-portion of the vertices in $Z\setminus N_Z(X_s)$. So, for significantly more than a third of the clusters of $R_G$  we have that a significant portion of  their vertices see at least $(\frac 12-\eps^{\frac 13}) \cdot |Z\setminus N_Z(X_s)|$ vertices in  $Z\setminus N_Z(X_s)$. Because of regularity, and because of~\eqref{cyro}, this means that for any such cluster $C$, all but at most an $\eps$-fraction of the vertices in $C_{W}$ has at least $(\frac 12-3\eps^{\frac 13}) \cdot |Z\setminus N_Z(X_s)|$ neighbours in  $Z\setminus N_Z(X_s)$.

Choose $S(s)$ as any one of the clusters as above, i.e.,  such that
\begin{enumerate}
\item[($\alpha$)] all but at most $\eps |S(s)_{\tilde W}|$ vertices of the set $S(s)_{\tilde W}$ have degree at least $(\frac 12-3\eps^{\frac 13}) \cdot |Z\setminus N_Z(X_s)|$  into  $Z\setminus N_Z(X_s)$; 
\end{enumerate}
and such that in addition (unless  $s=r$, in which case the following two conditions are void),
\begin{enumerate}
\item[($\beta$)] $S(s)$ is adjacent to $S(p(s))$; 
\item[($\gamma$)] $\varphi(p(s))$ is typical with respect to $S(s)_{\tilde W}$,
\end{enumerate}
where $S(p(s))$ denotes the cluster the parent $p(s)$ of $s$ was embedded into.
Such a choice of $S(s)$ is possible since by~\eqref{minRG}, cluster $S(p(s))$ has degree almost $\frac{2p}3$ in $R_G$, and because of~\eqref{goodparent}.

\subsubsection{Embedding seed $s$ into target cluster $S(s)$}\label{sec:embsintotarget}

 We place $s$ in a vertex $\varphi(s)$ from $S(s)_{\tilde W\setminus U}$ such that
 \begin{enumerate}[$(A)$]
 \item $\varphi(s)$ is a neighbour of $\varphi(p(s))$ (where $p(s)$ is the parent of $s$, and if $s=r$ this restriction is empty); \label{seedp(s)}
\item $\varphi (s)$ is typical to $C_{\tilde W}$ for all but at most $\eps p$ clusters $C\in V(R_G)\setminus S(s)$; \label{seedW}
\item $\varphi (s)$ is typical to $C_{\tilde V}$ for all but at most $\eps p$ clusters $C\in V(R_G)\setminus S(s)$; \label{seedtildeV}
\item $\varphi (s)$ is typical to $C_{F_1}\setminus (U\cup U')$ for all but at most $\eps p$ clusters $C\in V(R_G)\setminus S(s)$; and \label{seedF1}
 \item $\varphi (s)$ is typical to $C_L$ for all but at most $\eps p$ clusters $C\in V(R_G)\setminus S(s)$. \label{seedZ}
\end{enumerate}

Such a choice is possible  since by~\eqref{cyro}, almost all vertices in any given cluster are typical with respect to any fixed significant subsets of almost all other clusters. 
Note that in particular, $(E)$ implies that 
\begin{equation}\label{degreetwothirdstoZ}
deg_Z(\varphi (s))\geq (\frac 23-\eps^{\frac 13})|Z|.
\end{equation}

\subsubsection{Degrees of the embedded seeds into $Z$}\label{degpropembseeds}

The reason for our choice of $S(s)$ as a cluster fulfilling property ($\alpha$) from Subsection~\ref{sec:target} is that it allows us to accumulate degree into~$Z$. More precisely, if we consider a seed $s$ together with its set of relevant seeds~$X_s$, then we know that the union of their neighbourhoods in $Z$ is significantly larger than the neighbourhood of $s$ alone. Better still, the more vertices $X_s$ contains, the larger becomes our bound on this neighbourhood.

Let us make this informal observation more precise in the following claim.
\begin{claim}\label{properties_of_groups} Let $B$ be a group of seeds.
\begin{enumerate}[(i)]
\item\label{aaa} If $B$ has size five, then 
$$ |N(\varphi (B))\cap Z|\geq (\frac{47}{48}  -\epsilon^\frac{1}{4} )\cdot |Z|.$$ 
\item\label{bbb}
If $B$  has size  four and is of type 1, then 
$$|N(\varphi (B))\cap Z|\geq (\frac{23}{24} -\epsilon^\frac{1}{4} )\cdot |Z|.
$$
\item\label{ccc} If $B=\{b_1,b_2,b_3,b_4\}$ (with the seeds $b_i$ appearing in this order in $\sigma$) is of type 2, then
$$|N(\varphi (B))\cap Z|\geq (\frac{11}{12}  -\epsilon^\frac{1}{4} )\cdot |Z|,$$
and 
$$\min\big\{|N(\varphi (\{b_1, b_2\}))\cap Z|,|N(\varphi (\{b_3, b_4\}))\cap Z|\big\}\geq(\frac{5}{6}  -\epsilon^\frac{1}{4} )\cdot |Z|.$$
\item\label{ddd} If $B$ is large, say of size $47\cdot 2^j$, then $$|N(\varphi (\{x^B_i: i=1\ldots, j+6\}))\cap Z|\geq (1-\frac{1}{96\cdot 2^j} -\epsilon^\frac{1}{4} )\cdot |Z|,$$ where $(x^B_i)_{i=1\ldots, j+6}$ is the sequence defined in Subsection~\ref{sequencesandlevels}. 
\end{enumerate}
\end{claim}

\begin{proof}
This follows rather directly from ($\alpha$) and $(E)$ (from Subsections~\ref{sec:target} and~\ref{sec:embsintotarget}, respectively), from~\eqref{degreetwothirdstoZ}, and from the definition of the set $X_s$ of relevant seeds (Definition~\ref{relevantseeds}). For instance, we can calculate the bound in item $(i)$ by using \eqref{degreetwothirdstoZ}, ($\alpha$), $(E)$, and Definition~\ref{relevantseeds}~$(b)$ to see that $$|N(\varphi (B))\cap Z|\geq (\frac 23 + \frac 16+ \frac 1{12}+ \frac 1{24} + \frac 1{48} - 5\cdot 3\eps^{\frac 13})\cdot |Z|\geq (\frac{47}{48}  -\epsilon^\frac{1}{4} )\cdot |Z|.$$

For item $(iv)$, we need to take slightly more care with the calculation. Note that the degree into $Z$ of the image of the first seed is off $\frac 23|Z|$ by at most $3\eps^{\frac 13}|Z|$. The degree of second seed's image is only off $\frac 12|Z\setminus N_Z(x_1^B)|$ by less than $3\eps^{\frac 13}\frac{|Z|}2$. For the third seed we are only off $\frac 12|Z\setminus (N_Z(x_1^B)\cup N_Z(x_2^B))|$ by $3\eps^{\frac 13}\frac{|Z|}4$, and so on, which means we can  bound the error in our  estimate for the size of the joint neighbourhod in  $(iv)$ by $$(1+\frac 12+\frac 14+\ldots)\cdot 3\eps^{\frac 13}|Z|\le 2\cdot 3\eps^{\frac 13}|Z|\le \epsilon^\frac{1}{4} \cdot |Z|.$$
\end{proof}

\subsection{Embedding the small trees}\label{emb:trees}

Assume we have sucessfully embedded a seed $s$, and are now, before we proceed to the next seed, about to embed all small trees from $\mathcal T_s$ (while still leaving any leaves from $L$ adjacent to $s$ unembedded).

Our plan is to embed those trees of $\mathcal T_s$ that belong to $F_1$ into $\bigcup_{C\in V(R_G)} C_{F_1}$, and those trees of $\mathcal T_s$ that belong to $F_2$ into $\bigcup_{C\in V(R_G)} C_{F_2}$. We first explain how we deal with the larger trees, i.e.~those in $F_2\setminus F_2'$, and those in $F_2'$. After that we explain how we deal with the constant sized trees, i.e.~those in~$F_1$. Note that actually, it does not matter in which order we deal with the sets $F_1, F_2\setminus F_2', F_2'$.

\subsubsection{Embedding the trees from $F_2\setminus F_2'$}

For each $\bar T\in \mathcal T_s\cap (F_2\setminus F_2')$, let $r_{\bar T}$ denote its root. We plan to put $r_{\bar T}$ into $C'_{\tilde V}\setminus U$ for some suitable cluster $C'$. (We will explain below how exactly we do that.) For the rest of $V(\bar T)$, we proceed as follows. 

Recall that we defined a perfect matching of $R_G$ near the  end of Section~\ref{prepas}. 
Choose an edge $CD$ of $M_{F_2}$ that contains at least $3\eps\cdot \frac mp$ unused vertices in each of $C_{F_2}$, $D_{F_2}$. If $|C_{F_2}\setminus U|\geq |D_{F_2}\setminus U|$, we will aim at putting the larger colour class of $\bar T-r_{\bar T}$ into $C_{F_2}\setminus U$, and otherwise, we aim at putting it into $D_{F_2}\setminus U$. Observe that if we manage to do this for every tree $\bar T$ we embed, we can ensure that throughout the process (even when embeeding trees from~$\mathcal T_{s'}$, for some $s'\neq s$), the edges from $M_{F_2}$ keep their free space in a more or less balanced way, that is, for all edges $C'D'$ in $M_{F_2}$,
\begin{align}\label{balanced}
&\text{$|C'_{F_2}\setminus U|$ and $|D'_{F_2}\setminus U|$ differ by at most $\beta \frac mp$.}
\end{align}

Let us now explain how we manage to embed $\bar T$ in this way. Assume our aim is to embed the children of $r_{\bar T}$ in to $C_{F_2}\setminus U$, the grandchildren into $D_{F_2}\setminus U$, the grand-grandchildren into $C_{F_2}\setminus U$, and so on. The embedding of $\bar T-r_{\bar T}$ will be easy using Lemma~\ref{treeinpair} once we found a vertex $\varphi (r_{\bar T})$ to embed~$r_{\bar T}$ into, that is, a vertex that is both a neighbour of $\varphi(s)$ and typical with respect to $C_{F_2}\setminus U$.

So we only need to find a suitable vertex for $\varphi (r_{\bar T})$, the image of the  root of $\bar T$ (which belongs to $\tilde V$). In order to do so, we first determine a cluster~$C'$ that is adjacent to both $C$ and $S(s)$, and that fulfills $d(S(s), C')\geq \frac 1{4}$. At least nearly a third of the clusters in $R_G$ qualify for this, because of~\eqref{minRG}. Now, by~\eqref{seedtildeV} in the choice of $\varphi (s)$ (in Subsection~\ref{sec:embsintotarget}), we know that $\varphi (s)$ has typical degree into the set $C'_{\tilde V}$ for all but very few clusters $C'$. Typical degree means that $\varphi (s)$ has at least $(\frac 1{4}-\eps^2)\cdot |C'_{\tilde V}|$ neighbours in $C'_{\tilde V}$, and by Lemma~\ref{cutT}~\eqref{eqXX}, at most $2\beta m$ vertices have been used for earlier vertices from~$\tilde V$. So, by~\eqref{beta_eps_lambda_alpha}, we can choose a suitable $C'$ such that $\varphi (s)$ has a large enough neighbourhood in $C'_{\tilde V}$ to ensure it contains a vertex $\varphi (r_{\bar T})$ that is typical with respect to $C_{F_2}\setminus U$, i.e.~such that  $\varphi (r_{\bar T})$ has at least $(10\sqrt\eps-\eps^2)|C_{F_2}\setminus U|\ge \beta m$ neighbours in $C_{F_2}\setminus U$, where the inequality follows from ~\eqref{slicesF2} and~\eqref{balanced}.

\subsubsection{Embedding the trees from $F_2'$}\label{sec:tree_embed}

For each $\bar T\in \mathcal T_s\cap F_2'$,
we proceed exactly as in the preceding paragraph, except that now, we have to make a small adjustment when we are close to embedding $\tilde v$, the second vertex from $\tilde V$ contained in $V(\bar T)$. 

Suppose $s'$ is the seed which is adjacent to $\tilde v$ in $T$. Because we embed the seeds following the transversal order, we know that~$s'$ is not yet embedded by the time we deal with $\bar T$. We take care to embed $\tilde v$ into a vertex that is typical with respect to almost all the sets $C_{\tilde W}$. That is, the image of $\tilde v$ will be chosen such that~\eqref{goodparent} holds.

Finally, observe that~\eqref{slicesF2} and~\eqref{balanced} ensure that the space we had assigned to $F_2$ is enough for embedding first all of $F_2\setminus F_2'$, and now, all of $F'_2$.

\subsubsection{Embedding the trees from $F_1$}

We now explain how we embed the trees from $\mathcal T_s\cap F_1$. Note that because of~\eqref{slicesF1} and~\eqref{U'small}, we have enough space to embed all of $\mathcal T_s\cap F_1$. Furthermore, because the trees from $F_1$ are small, and because of regularity, we have no problem with the actual embedding of them into the regular pairs of $G$. The only thing we need to make sure is that the roots of the trees from $\mathcal T_s\cap F_1$ are embedded into neighbours of $\varphi(s)$, and that we maintain the unused parts of the cluster slices $C_{F_1}$ balanced at all times.

Since there is no matching like $M_{F_2}$ that can be used throughout the whole embedding (i.e., for all seeds), we will have to simultaneously keep  {\em all of the clusters} reasonably balanced. This will be possible  because of the rather delicate embedding strategy we employ, and which we will start to explain now. 
 
 \paragraph{\noindent Preparing the slices $C_{F_1}$.} Assume we are about to start the embedding  process of the trees from $\mathcal T_s\cap F_1$. First of all, note that we can partition the free space $C_{F_1}\setminus (U\cup U')$ of the slices $C_{F_1}$ of each of the clusters $C\in V(R_G)\setminus\{S(s)\}$  into sets $Q_0^C, \ldots , Q_r^C$ for some~$r$, such that $|Q_0^C|<2\lceil\eps \frac mp\rceil$ and $|Q_i^C|=\lceil\eps \frac mp\rceil$ for $i=1, \ldots , r$, and such that for each $i=1, \ldots ,r$, either all or none of the vertices in $Q_i^C$ are adjacent to $\varphi (s)$.  
The reason for partition the free space in this way is that we now have total control over where exactly the neighbours of $\varphi (s)$ are (since  the sets $Q_0^C$ are small enough to be ignored during this step of the embedding).
Observe that  the sets $Q_i^C$,  for $i=1, \ldots , r$, are still large enough to preserve the regularity properties, although now, in view of (R2) from Section~\ref{regularity}, we have to replace the regularity parameter $\eps^2/2$ with $2\frac{\eps^2/2}\eps=\eps$. 

Consider the graph $H$ with vertex set $\{Q_i^C\}_{i=1,\ldots ,r, C\in V(R_G)}$ and an edge for each $\eps$-regular pair of sufficient density. Say $H$ has $p'$ vertices. By~\eqref{minRG}, the weighted  minimum degree of  $R_G$ is bounded by $\delta_w (R_G)\geq (\frac 23-13\eps)p$, and therefore, the weighted  minimum degree of  $H$ is bounded by $\delta_w (H)\geq (\frac 23-17\eps)p'$. 

So, by our choice of $\varphi (s)$, in particular by $(D)$ of Subsection~\ref{sec:embsintotarget}, we know that $\varphi (s)$ is typical to almost all clusters, and thus has neighbours in at least $(\frac 23-20\eps)p'$ of the sets $Q_i^C$.
Let $N$ consist of a set of $\lceil(\frac 23-20\eps)p'\rceil$ sets $Q_i^C$ that contain neighbours of $\varphi (s)$. 
We apply Lemma~\ref{lem:match_paths} with $\xi:= 17\eps$ to $H$  to obtain a set $X$, of size 
\begin{equation}\label{sizeofX}
|X|\le \lfloor 391 \eps p'\rfloor +3\le 400\eps  p',
\end{equation}
 as well as  an $(N\setminus X)$-good matching $M$ ($=M_s$), an $(N\setminus X)$-in-good path partition $\mathcal P_A$ and an $(N\setminus X)$-out-good path partition $\mathcal P_B$ of $H-X$.

Set $\mathcal Q:=\bigcup_{C\in V(R_G)}\{Q_1^C, Q_2^C,\ldots,Q_r^C\}\setminus X$.
By~\eqref{slicesF1}, and by~\eqref{U'small}, we know that $\bigcup\mathcal Q$ is large enough to host all of $\bigcup\mathcal T_s\cap F_1$. In fact, if $\bigcup\mathcal T_s\cap F_1$ could be embedded absolutely balanced into the sets $Q\in \mathcal Q$, then there would even be a leftover space of more than $100\eps\frac mp$ in each of the sets~$Q$.

Recall that during the embedding of the trees from $\mathcal T_s\cap F_1$, we will add some vertices to a set $U'$, for keeping better track of the balancing of the edges. We will keep $U'$ small, that is, we will ensure that~\eqref{U'small} holds.

\paragraph{\noindent Preparing $\mathcal T_s\cap F_1$.} We now partition the set of trees from $\mathcal T_s\cap F_1$ into three sets\footnote{We remark that it is not really necessary to treat the trees from $T_{Bal}$ separately (as they could be treated together with the trees from $T_{Unbal}$ in Phase 2), but we believe that embedding $\bigcup T_{Bal}$ first (in Phase 1) is more instructive.}: the set $T_{Bal}$  contains all the balanced trees, i.e.~those trees whose color classes have the same size; the set $T_{NearBal}$ contains all trees having the property that their colour classes differ by exactly one, with the bigger class containing the root; and the set $T_{Unbal}$  contains all the remaining trees, that is all unbalanced trees  not belonging to $T_{NearBal}$. 

\paragraph{\noindent Phase 1.} In the first phase of our embedding, we embed all trees from $T_{Bal}$, using the matching $M$. We try to spread these trees as evenly as possible among the edges of $M$. It is not difficult to see that by Lemma~\ref{cutT}~\eqref{tiny}, it is possible to make the used part of the clusters differ by at most $\frac 1\beta$ (but even the more obvious weaker bound $\frac 1{\beta^2}$ is sufficient for our purposes). At the end of this phase of the embedding, we add to $U'$ at most $\frac{1}\beta$ unused vertices from each of the clusters $Q\in V(M)$, and can thus make sure each of the clusters has exactly the same number of vertices in $Q_{F_1}\setminus (U\cup U')$. 
 
\paragraph{\noindent Phase 2.} In the second phase of our embedding, we embed all trees from $T_{Unbal}$. We group the trees from $T_{Unbal}$ by their number of vertices, which is some number between $3$ and $\frac{1}\beta$. Then we subdivide these groups according to the number of vertices belonging to the same colour class as their root.
 The final groups represent the {\it types} of trees. Since all trees we consider have order at most $\frac 1\beta$, 
 \begin{equation}\label{howmanytypes}
 \text{there are at most $\frac 1{\beta^2}$ different types of trees in $T_{Unbal}$.}
 \end{equation}
 
For each of the types $\bar T$, say with $t$ vertices, and colour classes of sizes $t_1$ and $t_2$, where the class of size $t_1$ contains the root, we proceed as follows. We go through the elements of our $N$-out-good path partition $\mathcal P_B$ in some fixed order, always embedding only a constant number of trees of type~$\bar T$. In each round, we keep the clusters of $H-X$ perfectly balanced. Only when we run out of trees of type $\bar T$, we will (necessarily) have to make a last round, possibly not reaching all elements of $\mathcal P_B$, and thus unbalancing some of the clusters a bit (by at most $\frac 1\beta$).

To make the above description more precise, recall that $\mathcal P_B$ consists of
\begin{enumerate}[(M1)]
\item single edges $AB$ with both ends in $N$;
\item paths $ABCD$ with $B, C\in N$; and
\item paths $ABCDEF$ with $B, C, D, E\in N$.
\end{enumerate}
Say there are $m_1$ paths $AB$ as in (M1), $m_2$ paths $ABCD$ as in (M2), and $m_3$ paths $ABCDEF$ as in (M3).
Let us now analyse how the sets~$Q\in\mathcal Q$ lying in edges or paths from (M1)--(M3) fill up when we embed small trees of type $\bar T$ into them in the following specific ways.

\paragraph{Paths as in (M1).}First, the sets $A$, $B$ of any edge as in (M1) will each get filled up with~$t$ vertices if we embed one tree of type $\bar T$ in one `direction' and a second tree of type $\bar T$ in the other `direction'. In other words,  we can embed a total of $2m_1h$ trees  of type $\bar T$  into the edges from (M1), filling each of the corresponding clusters $Q$  with $th$ vertices (where $h$ is any not too large natural number).

\paragraph{Paths as in (M2).}The sets~$Q$ on paths $ABCD$ as in (M2) will get filled
as follows.  If we 
\begin{itemize}
\item perform $x$  rounds in which we embed one tree of the current type  $\bar T$ in the edge $AB$, with the root going to $B$; 
\item perform $x$  rounds in which we embed a tree of type $\bar T$ in the edge $CD$, with the root going to $C$; 
\item perform $y$  rounds  of embedding a tree of type $\bar T$ with the root going to~$C$, but the rest of the tree going to $AB$; and
\item perform $y$  rounds  of embedding a tree of type $\bar T$ with the root going to~$B$, but the rest of the tree going to $CD$,
\end{itemize}
 then after these $2x+2y$ rounds,  $A$ and $D$ each have received $xt_2+y(t_1-1)$ vertices, while  $B$ and $C$ each have received $xt_1+y(t_2+1)$ vertices. 

So, if $t_1>t_2$ (observe that then actually $t_1\geq t_2+2$, since $\bar T\notin T_{NearBal}$), we will have filled each of the four sets $A, B, C, D$ with exactly $(t_1-t_2-1)t$ vertices if we choose $x=t_1-t_2-2$ and $y=t_1-t_2$. If $t_2\geq t_1$, we can fill each of the four sets $A, B, C, D$ with exactly $(t_2-t_1+1)t$ vertices by taking $x=t_2-t_1+2$ and $y=t_2-t_1$. 

Resumingly, for any not too large natural number $h'$, there is a way to embed a total of $|t_1-t_2-1|\cdot 4m_2 h'$ trees of type $\bar T$ into the edges from (M2), filling up each of the corresponding clusters $Q$  with $|t_1-t_2-1|\cdot t$ vertices. Even more, taking into account what we said above for edges from (M1), we conclude that for any not too large $h'$, we are able to embed a total of $$ |t_1-t_2-1| \cdot 2m_1h' + |t_1-t_2-1|\cdot 4m_2 h'=|t_1-t_2-1| \cdot (2m_1 + 4m_2) h'$$ trees of type $\bar T$ into the edges from (M1) and (M2), filling up each of the corresponding clusters $Q$  with $|t_1-t_2-1|\cdot t h'$ vertices. 

\paragraph{Paths as in (M3).}
For the paths $ABCDEF$ from (M3) we can calculate similarly: Say we do $x$ rounds of  embedding of a tree of type~$\bar T$ in the edge $AB$ and another $x$ rounds embedding it into $EF$. We then do~$y$ rounds of embedding the tree into  $AB$, but with  the root of the tree going into $C$, and another $y$ rounds putting it into $EF$, with the root going into $D$. Moreover, we perform $2z$ rounds where we embed the tree into $CD$, of which $z$ rounds in each `direction'. Then after these $2x+2y+2z$ rounds, we filled each of   $A$ and $F$ with 
$xt_2+y(t_1-1)$ vertices, each of  $B$ and $E$ with 
$xt_1+yt_2$ vertices, and each of  $C$ and $D$ with 
$y+zt$ vertices.

So, if $t_2\geq t_1$, then with $x=t\cdot(t_2-t_1+1)$, $y=t\cdot(t_2-t_1)$, and $z=(t-1)\cdot(t_2-t_1)+t_1$, we have filled each of the six sets $A, B, C, D, E, F$ with exactly the same amount of vertices, namely with $t\cdot(t\cdot(t_2-t_1)+t_1)$ vertices each. If $t_1>t_2$, we choose $x=t\cdot(t_1-t_2-1)$, $y=t\cdot(t_1-t_2)$, and $z=(t-1)\cdot(t_1-t_2)-t_1$, and fill each   of the six sets $A, B, C, D, E, F$ with  $t\cdot(t\cdot(t_1-t_2)-t_1)$ vertices. So, adopting the convention that $\pm t_1$ means $+t_1$ if $t_2\geq t_1$ and $-t_1$ otherwise, we can embed, for any not too large $h''$,  a total of $6(t\cdot |t_2-t_1|\pm t_1)m_3h''$ trees of type $\bar T$ into the edges from (M3), placing exactly $t\cdot(t\cdot |t_2-t_1|\pm t_1)h''$ vertices into each of the corresponding clusters~$Q$.

Recalling the earlier observations on embedding trees of type $\bar T$ into edges from (M1) and (M2), we conclude that
we are able to embed 
\begin{align*}
d_{\bar T}  :=  |t_1 & -t_2-1|  \cdot (2m_1 + 4m_2) (t\cdot |t_2-t_1|\pm t_1)\\ &+  6(t\cdot |t_2-t_1|\pm t_1)m_3\cdot |t_1-t_2-1|\\  
= |t_1 & -t_2-1|  \cdot  (t\cdot |t_2-t_1|\pm t_1) \cdot (2m_1 + 4m_2 + 6m_3)
\end{align*}
trees of the current type $\bar T$, using all sets~$Q\in\mathcal Q$  in a  completely balanced way (each receives exactly $t\cdot(t\cdot |t_2-t_1|\pm t_1)\cdot |t_1-t_2-1|$ vertices).

Now, we embed the first $d_{\bar T}$ trees of type $\bar T$ in this way, then proceed to embed the next $d_{\bar T}$ trees of this type, then the next $d_{\bar T}$ such trees, and so on. If at some point (this might happen in the first round already), there are less than $d_{\bar T}$ trees of type $\bar T$ left, then we perform a last round for embedding these  trees, simultaneously blocking at most $d_{\bar T}\cdot t$ vertices which we add to~$U'$ (note that they have not been used for the embedding). In this way, we can finish the embedding of all the trees of the current type  $\bar T$ while perfectly balancing $Q\cap (U\cup U')$ for all clusters $Q\in\mathcal Q$. In particular, each of the clusters has exactly the same number of vertices in $Q_{F_1}\setminus (U\cup U')$. 

 Note that, when working on one $t$-vertex tree $\bar T\in T_{Unbal}$, the number of vertices we add to $U'$ although they are not actually used for the embedding  is at most  $$d_{\bar T}t\le 2t^4\cdot p'\le \frac2{\beta^4}\cdot \frac p\eps\le \frac{1}{50\beta^5}\le \frac{1}{\beta^5},$$ where we used that each tree in $T_{Unbal}\subseteq F_1$ has at most $\frac 1\beta$ vertices. So,  the number of vertices we add to $U'$ after working on all trees from $T_{Unbal}$ is at most the number of types of trees multiplied by $\frac{1}{\beta^5}$, and thus, by~\eqref{howmanytypes}, at most $ \frac 1{\beta^{10}}$.
 
\paragraph{\noindent Phase 3.} In this phase, we embed the trees from $T_{NearBal}$. Each of these trees  has (at least) one leaf in its heavier colour class. Instead of the root, as in phase 2, we will now put this leaf 
 into a different cluster, and instead of the $N$-out-good path partition we will be using the $N$-in-good path partition~$\mathcal P_A$.
 
Again we go through the different types $\bar T$ of trees, of which there are at most $\frac 1{2\beta}$. Now say we are working on the trees of a fixed type $\bar T$, with $t$ vertices.
 Easier considerations than in the previous case show that we can
embed exactly $t$ vertices into each of the slices $Q_{F_1}$ of clusters $Q\in V(M)$ if we  embed  six trees of type $\bar T$ into the six clusters corresponding to a path of length 6, four trees of type $\bar T$ into the four  clusters corresponding to a path of length 4, and two trees of type $\bar T$ into the  clusters corresponding to a path of length 2. 
 So, putting at most $\frac{1}{2\beta}\cdot \frac{1}{\beta}\cdot p'\leq \frac{1}{\beta^3}$ unused vertices into~$U'$, we can finish the embedding of all the trees of $T_{NearBal}$ balancing all slices as desired. In particular, each of the clusters $Q\in\mathcal Q$ has exactly the same number of vertices in $Q_{F_1}\setminus (U\cup U')$. 
 
 \smallskip
 
 After finishing Phase 3, we still put some more vertices into $U'$, before declaring the embedding procedure of the trees in $\mathcal T_s\cap F_1$ finished. Namely, we put an appropiate number of vertices from the sets  in $X$ into~$U'$. That is, the number of vertices from any of the sets of $X$ we add to $U'$ is the same as the number of vertices from any of the sets $Q\in\mathcal Q$ that went to $U$ or to $U'$ during the embedding of $\mathcal T_s\cap F_1$. This cleaning-up is only done because it will be nicer to be able to start the embedding of the trees at the next seed with all slices $C_{F_1}$ perfectly balanced.
 
 Observe that the number of vertices we added to $U'$ while dealing with the trees from $\mathcal T_s\cap F_1$ is at most\footnote{Note that we did not add $\bigcup_{C\in V(R_G)} (Q_0^C\cup Q_1^C)$ to $U'$.}
 $$ u'_s\ \leq \ \frac 1\beta +\frac 1{\beta^{10}}+\frac{6}{\beta^3}+ |X|\cdot \frac{2|\mathcal T_s\cap F_1|}{p'}\ \leq \ \frac 3{\beta^{10}}+800\eps \cdot |\mathcal T_s\cap F_1|,$$
 where we used~\eqref{m_delta} and~\eqref{sizeofX} for the last inequality.
Hence the bound~\eqref{U'small} we had claimed above is correct. This ensures we have enough space for all future trees from $F_1$.

\subsection{Embedding the leaves}\label{emb:leaves}

This section is devoted to the embedding of the leaves. That is, we are now at a stage where we have sucessfully embedded all seeds and all small trees, and all that is left to embed is $L$, the set of leaves adjacent to seeds. We will show we can embed all of $L$ at once.

If we cannot embed $L$ into $Z$, then by Hall's theorem\footnote{Hall's theorem can be found in any standard textbook, it states that a bipartite graph with bipartition classes $A$ and $B$ either has a matching covering all of $A$, or there is an `obstruction': a set $A'\subseteq A$ such that $|N(A')|<|A'|$.}, there is some subset $K\subseteq \tilde W$ such that 
\begin{equation}\label{hallKLK}
|N(\varphi (K))\cap Z|<|L_K|,
\end{equation}
 where $L_K$ is the set of leaves adjacent to elements of $K$, the set $\varphi (K)$ is the set of images of $K$, and $N(\varphi (K))\cap Z$ is the union of the neighbours  in~$Z$  of the elements of $\varphi (K)$. 
 
 Recall that by~\eqref{seedZ} from Subsection~\ref{sec:embsintotarget}, we chose as the image of a seed~$s$ a vertex $\varphi(s)$ that is typical with respect to $C_L$ for almost all clusters $C$ of~$R_G$. 
 Because of~\eqref{minRG}, this means that
\begin{equation}\label{Ksees2/3}
\text{ each element of $\varphi (K)$ sees at least $(\frac{2}{3}-20\epsilon)|Z|$ 
vertices of $Z$.}
\end{equation}
 In particular, by~\eqref{sizeofZforleaves}, each element of $\varphi (K)$ sees more than $\frac{5}{8} (|L|+ \frac{9}{10} \alpha m)$ vertices of $Z$. Thus, we may assume that 
 \begin{equation}\label{Lgeq3alpham2}
 |L|>\frac{3\alpha m}{2},
 \end{equation}
  as otherwise $\frac{5}{8} (|L|+ \frac{9}{10} \alpha m)\geq |L|$, which means that we could have embedded~$L$ without a problem. 
  
  
  Our aim is to reach a contradiction to the assumption that the set $K$ exists. We will reach this contradiction by proving in Claims~\ref{nolargegroupinK}--\ref{nosize4type2groupmeetsKin1} that~$K$ misses a vertex in each of the large groups and also in most of the small groups of seeds we defined in Subsection~\ref{groupandorder}. In some of the small groups $K$ actually misses more than one vertex. We will prove these claims  by repeatedly using~\eqref{hallKLK}. 
  This means that in total, $K$ misses many vertices from $\tilde W$, and these vertices spread out among the blocks (and thus have a corresponding proportion of the leaves hanging from them). Therefore, we can conclude that $|L_K|$ is smaller than the bound for the neighbourhood of $\varphi (K)$ given in~\eqref{Ksees2/3}, and thus, $L_K$ could have been embedded without a problem, which is a contradiction. 

  
Let us make this outline more precise.
We start by proving that each of the large groups has a vertex outside $K$.
\begin{claim}\label{nolargegroupinK}
No large group is completely contained in $K$.
 \end{claim}
 \begin{proof}
Assume otherwise, and consider the largest $j$ for which there is a group of size 
$47\cdot 2^j$ completely contained in $K$. Then by Claim~\ref{properties_of_groups}~(iv) from Subsection~\ref{degpropembseeds},  we know that 
\begin{equation}\label{laaarge}
|N(\varphi (K))\cap Z|\geq (1-\frac{1}{95\cdot 2^j} -\epsilon^\frac{1}{4} )|Z|.
\end{equation}
 If $j\geq j^\circledast :=\lceil \log\frac 1{95\cdot\frac{999}{1000}\alpha}\rceil$, then by~\eqref{sizeofZforleaves}, the bound from~\eqref{laaarge} exceeds $|L|$, which yields a contradiction to~\eqref{hallKLK}. So \begin{equation}\label{specialj}j< j^\circledast.\end{equation}

In particular, because of~\eqref{beta_eps_lambda_alpha} and~\eqref{defj^*}, we know that $j<j^*$, and hence, there exists a large group  of size $47\cdot 2^{j+1}=94\cdot 2^j$. For any such group $B$, we know that, by the choice of $j$, there is a vertex $v_{B}\in B$ that is not in~$K$. Let $L_{B}$ be the set of leaves adjacent to seeds in $B$, and let $\ell_{\max} (B)$ and $\ell_{\min} (B)$ be the number of leaves adjacent to the first and the last seed in $B$, respectively, under the size order $\sigma$. Thus, every seed $b\in B$ is adjacent to a number $\ell_b$ of  leaves, with $\ell_{\min} (B)\leq \ell_b\leq \ell_{\max} (B).$

Set  $${\it dif} (B):=\ell_{\max} (B)-\ell_{\min} (B).$$ Then 
\begin{align*}
\ell_{v_{B}}\geq \ell_{\min} (B)& \geq \frac{|L_{B}|}{|B|}-{\it dif} (B)\cdot\frac{|B|-1}{|B|}\\ & =\frac{|L_{B}|}{94\cdot 2^j}-{\it dif} (B)\cdot (1-\frac{1}{94\cdot 2^j}).
\end{align*}
Since the  groups are consecutive in the size order $\sigma$, and since no seed has more than $\alpha m$ leaves adjacent to it, we know that 
$$\sum_{B: \ \text{$|B|=94\cdot 2^j$}}{\it dif} (B)\leq \alpha m.$$
So,  the  number of leaves adjacent to seeds that are not in $K$ can be bounded by calculating
\begin{align*}
|L|-|L_K| & \geq
\sum_{B: \ \text{$|B|=94\cdot 2^j$}}\ell_{v_{B}}\\ & \geq \frac{|L|}{94\cdot 2^j}-\sum_{B: \ \text{$|B|=94\cdot 2^j$}}{\it dif} (B)\cdot (1-\frac{1}{94\cdot 2^j})\\
& \geq \frac{|L|}{94\cdot 2^j}-\alpha m \cdot (1-\frac{1}{94\cdot 2^j}).
\end{align*}
Therefore,
\begin{align}
|L_K|& \leq (1-\frac{1}{94\cdot 2^j})\cdot (|L|+\alpha m)\notag \\
&\leq (1-\frac{1}{94\cdot 2^j}) \cdot (|Z|+\frac{\alpha^4}{10^6}m)\notag \\
& \leq (1-\frac{1}{95\cdot 2^j} -\epsilon^\frac{1}{4} )|Z|\notag \\ & \leq |N(\varphi (K))\cap Z|,\label{itacare}
\end{align}
where the second and last  inequalities follow from~\eqref{sizeofZforleaves} and~\eqref{laaarge}, respectively, and the third inequality follows from the observation that
\[
\frac{\alpha^4}{10^6}m \ \leq \ \frac 23\cdot\frac{\alpha^3}{10^6}|Z|\ \leq \ (\frac{1}{94\cdot 95\cdot 2^j}-\eps^{\frac 14})|Z|,
\]
where for the first inequality we used that $|Z|\geq |L|>\frac{3\alpha m}2$ (by~\eqref{Lgeq3alpham2}), and the second inequality follows from
the facts that  $\eps\leq\frac{\alpha^4}{10^{18}}$ (by~\eqref{alpha_eps_beta}) and $j\leq j^\circledast$ (by~\eqref{specialj}).

Now, inequality~\eqref{itacare} gives a contradiction to~\eqref{hallKLK}. This proves Claim~\ref{nolargegroupinK}.
\end{proof}

Next, we show a similar fact for all small groups of size five.
\begin{claim}\label{nosize5groupinK}
No small group of size $5$ is completely contained in $K$.
 \end{claim}
\begin{proof}
Indeed, otherwise, because of Claim~\ref{properties_of_groups}~(i), we know that 
\begin{equation}\label{laaaaarge}
|N(\varphi (K))\cap Z|\geq (\frac{47}{48}  -\epsilon^\frac{1}{4} )|Z|.
\end{equation}
 Moreover, Claim~\ref{nolargegroupinK} implies that every large group $B$
of size 47 has a vertex~$v_{B}$ which is not in $K$, and thus we can calculate, similar as above for Claim~\ref{nolargegroupinK}, that
$$|L|-|L_K| \ \geq
\sum_{B: \ \text{$|B|=47$}}\ell_{v_{B}} \ \geq \ \frac{|L|}{47}-\alpha m \cdot (1-\frac{1}{47}),$$
and thus, employing~\eqref{sizeofZforleaves}, we find that
$$|L_K| \leq (1-\frac{1}{47}) \cdot (|Z|+\frac{\alpha^4}{10^6}m),$$
which, with the help of~\eqref{Lgeq3alpham2} and~\eqref{laaaaarge}, and using the fact that $\alpha \gg\eps$, yields a contradiction  to~\eqref{hallKLK}.
 This proves Claim~\ref{nosize5groupinK}.
\end{proof}

Next, we turn to the groups of size four, separating the treatment of these into two cases depending on their type.
\begin{claim}\label{nosize4type1groupinK}
No small group of size $4$ and of type 1 is completely contained in $K$.
 \end{claim}
 \begin{proof}
 Otherwise, because of Claim~\ref{properties_of_groups}~(ii), we know that 
\begin{equation}\label{laaaaaaarge}
|N(\varphi (K))\cap Z|\geq (\frac{23}{24}-\epsilon^\frac{1}{4} )|Z|.
\end{equation}
By Claim~\ref{nosize5groupinK} we know that every group of size five has a vertex
 which is not in $K$. Hence, every large group $B$ of size $47$ contains at least two vertices $v_B^1$ and~$v_B^2$ that are not in $K$. Moreover, by the definition of the small groups, the first of these vertices, $v_B^1$,  is one of the first $23$ vertices of $B$ under the size order,  and the second vertex $v_B^2$,  is one of the next $23$ vertices of $B$ under the size order. 
 
 So, we can 
 split the group $B$ minus its last vertex into two groups $B_1, B_2$ containing the first $23$ and the next $23$ consecutive elements in the size order, respectively,  with $v_B^1\in B_1$ and $v_B^2\in B_2$.
 Defining ${\it dif} (B_1)$, ${\it dif} (B_2)$  as in Claim~\ref{nolargegroupinK} for each of these two subgroups $B_1, B_2$ of the group $B$ of size $47$, 
  and letting $\ell_{v_{B}^i}$ denote the number of leaves hanging from $v_{B}^i$, for $i=1,2$, we can calculate that
  \[
\ell_{v_{B}^1} \geq \frac{|L_{B_1}|}{23}-\frac{22}{23}\cdot {\it dif} (B_1),
\]
and 
\[
\ell_{v_{B}^2} \geq \frac{|L_{B_2}|}{23}-\frac{22}{23}\cdot {\it dif} (B_2),
\]
where $L_{B_1}$ and $L_{B_2}$ denote the sets of leaves adjacent to vertices from $B_1$ and~$B_2$, repectively.
Thus, letting $L_{B_3}$ denote  the set of leaves adjacent to the very last vertex of the group $B$ (of size $47$), and noticing that $|L_{B_3}|\leq\frac{|L_B|}{47}$, we can calculate that 
\begin{align*}
 |L|-|L_K| & \geq
\sum_{B: \ \text{$|B|=47$}}(\ell_{v_{B}^1}+\ell_{v_{B}^2})\\ & \geq  \frac{|L|-\sum_{B: \ \text{$|B|=47$}}|L_{B_3}|}{23}-\frac {22}{23}\cdot \sum_{B: \ \text{$|B|=47$}}({\it dif} (B_1)+{\it dif} (B_2))\\
& \geq \frac{|L|}{23}-\frac{|L|}{23\cdot 47}-\frac {22}{23}\alpha m.
\end{align*}
Therefore, using~\eqref{sizeofZforleaves} and~\eqref{laaaaaaarge}, we obtain that
\begin{align*}
|L_K|& \leq \frac{22}{23} (|L| +\alpha m)\cdot (1+\frac{1}{22\cdot 47})\notag \\
&\leq \frac{22}{23}(|Z|+\frac{\alpha^4}{10^6}m) \cdot (1+\frac{1}{22\cdot 47})\notag \\
& \leq (\frac{23}{24} -\epsilon^\frac{1}{4} )|Z|\notag \\ & \leq |N(\varphi (K))\cap Z|,
\end{align*}
 a contradiction to~\eqref{hallKLK}.  This proves Claim~\ref{nosize4type1groupinK}.
\end{proof}

\begin{claim}\label{nosize4type2groupinK}
No small group of size $4$ and of type 2 is completely contained in $K$.
 \end{claim}
 \begin{proof}
 Otherwise, because of Claim~\ref{properties_of_groups}~(iii), we know that 
\begin{equation}\label{laaaaaaaaaaarge}
|N(\varphi (K))\cap Z|\geq (\frac{11}{12} -\epsilon^\frac{1}{4} )|Z|.
\end{equation}
However, by Claims~\ref{nosize5groupinK} and~\ref{nosize4type1groupinK}, we know that every small group of size~5 and every small group of size 4 and of type 1 has a vertex  which is not in~$K$. So, we can split every large group $B$ of size 47 into five subgroups $B_1$, $B_2$, $B_3$, $B_4$, $B_5$, each consecutive in the size order, and with $|B_i|=11$ for $i=1,2,3,4$ and $|B_5|=3$, such that each of $B_1$, $B_2$, $B_3$, $B_4$ contains a vertex $v_{B}^1, v_{B}^2, v_{B}^3, v_{B}^4\notin K$. 

Similar as above, we can calculate that
\begin{align*}
 |L|-|L_K| & \geq
\sum_{B: \ \text{$|B|=47$}}(\ell_{v_{B}^1}+\ell_{v_{B}^2}+\ell_{v_{B}^3}+\ell_{v_{B}^4})\\ & \geq  \frac{|L|}{11}- \frac{3|L|}{11\cdot 47}-\frac {10}{11}\alpha m,
\end{align*}
where numbers $\ell_{v_{B}^i}$ are defined as in the previous claim.
Now, using~\eqref{sizeofZforleaves} and~\eqref{laaaaaaaaaaarge}, we obtain that
\begin{align*}
|L_K|& \leq \frac{10}{11} (|L| +\alpha m)\cdot (1+\frac{1}{10\cdot 15})\notag \\
& \leq (\frac{11}{12}  -\epsilon^\frac{1}{4} )|Z|\notag \\ & \leq |N(\varphi (K))\cap Z|,
\end{align*}
 a contradiction to~\eqref{hallKLK}.
This proves Claim~\ref{nosize4type2groupinK}. 
\end{proof}

Next, we will  show that we can actually get some more out of the groups of type 2.
\begin{claim}\label{nosize4type2groupmeetsKin1}
No small group of size $4$ and of type 2 has three or more vertices in  $K$. 
 \end{claim}
 \begin{proof}
Indeed, 
 otherwise, because of the second part of Claim~\ref{properties_of_groups}~(iii), we know that
 \begin{equation}\label{laaaaaaaaaaaaarge}
|N(\varphi (K))\cap Z|\geq (\frac{5}{6} -\epsilon^\frac{1}{4} )|Z|.
\end{equation}
By Claims~\ref{nosize5groupinK}, ~\ref{nosize4type1groupinK} and~\ref{nosize4type2groupinK}, every small group $B$, except possibly those of size one, has a vertex $v_B$ which is not in $K$. So, similar as in the previous claims, but now going over all groups of sizes 4 and 5, and temporarily considering the  groups of size 1 to form part of the previous group (which had size 4, and so now has size 5), we  calculate that
\begin{align*}
 |L|-|L_K|  \geq  \sum_{B:\ 4\le |B|\le 5}\ell_{v_B}
 &  \ge \sum_{B:\ 4\le |B|\le 5}(\frac{|L_{B}|}{|B|}-{\it dif} (B)\cdot\frac{|B|-1}{|B|})\\
 & \ge \min\{\frac 14, \frac 15\}\cdot |L|-\max\{\frac 34, \frac 45\}\cdot\alpha m
 \\
 & \ge
  \frac{|L|}{5}
 -\frac {4}{5}\alpha m,
\end{align*}
where $L_B$ is the set of leaves incident with vertices of $B$, and ${\it dif}(B)$ is the difference between the biggest number of leaves hanging from a vertex in $B$ and the smallest such number.
We then use~\eqref{sizeofZforleaves} and~\eqref{laaaaaaaaaaaaarge}, to obtain that
\begin{equation*}
|L_K| \leq \frac{4}{5}( |L| +\alpha m)\leq (\frac{5}{6} -\epsilon^\frac{1}{4} )|Z| \le  |N(\varphi (K))\cap Z|,
\end{equation*}
 a contradiction to~\eqref{hallKLK}. This proves Claim~\ref{nosize4type2groupmeetsKin1}. 
\end{proof}

Resumingly, Claims~\ref{nosize5groupinK}--\ref{nosize4type2groupmeetsKin1} tell us that $K$ misses at least one vertex of each small group of size four or five, and misses at least two vertices from each small group of size four and type 2. Recalling our ordering  $4, {\bf 4}, 4, 4,  5, 4, {\bf 4}, 4, 4,  5, 4, 1$ of  the small groups inside each group $B$ of size $47$ as given in~\eqref{seq} in Subsection~\ref{ordering} (under ordering $\sigma$), we see that we can split $B$ into five groups $B_1,  B_2, B_3, B_4, B_5$ such that
\begin{itemize}
\item for $i=1,3$, the group  $B_i$ has $8$ vertices, at least $3$ of which are not in~$K$;
\item for $i=2,4$, the group  $B_i$ has $13$ vertices,  at least $5$ of which are not in~$K$; and 
\item $B_5$ has $5$ vertices, at least two of which are not in $K$.
\end{itemize}
Therefore, similar  calculations as for the previous claim give that
\begin{align*}
|L|-|L_K|\ & \geq \ \min\{ \frac{3}{8}, \frac 5{13}, \frac 25 \}\cdot |L|-\max\{ \frac{5}{8}, \frac 8{13}, \frac 35 \}\cdot \alpha m \ \\ & \geq \ \frac{3|L|}{8}-\frac {5}{8}\alpha m,
\end{align*}
and thus by~\eqref{sizeofZforleaves}, we get
\begin{equation*}
|L_K| \leq \frac{5}{8} \big (|L| +\alpha m\big )< \big (\frac 23 -20\eps \big)|Z|,
\end{equation*}
 a contradiction to~\eqref{Ksees2/3}.
This means the Hall-obstruction $K$ cannot exist, and we can thus finish the embedding of $T$ by embedding all leaves from $L$ in one step.
This finishes the proof of Lemma~\ref{highleafdegreecase}.

\section{Extending a given embedding}\label{sec:last}

For the companion paper~\cite{brucemaya2}, which contains the proof of the exact version of Theorem~\ref{chighleafdegreecase}, we will not only need Lemma~\ref{highleafdegreecase}, but also a second result, namely Lemma~\ref{maya1}, the main result of this section, which is stated below. Both lemmas are  very similar. 

The main difference in Lemma~\ref{maya1}  is that in the context of~\cite{brucemaya2},  a small tree~$T^*$ is already embedded, except for a small set~$Y\subseteq V(T^*)$. As images of neighbours of  $Y$ are well chosen,   we will later  be able to absorb~$Y$, i.e., we choose a suitable set $S\subseteq V(G)$ somewhat smaller than $Y$, embed $T-(T^*-Y)$ into $G-S$, and then complete the embedding by using the leftover plus $S$ for the embedding of $Y$.
 So, in Lemma~\ref{maya1}, we wish to embed $T-T^*$, and just as in Lemma~\ref{highleafdegreecase} we have some extra free space, but  now we have to cope with the already embedded $T^*-Y$, which may block neighbourhoods.  However, as we will see below,  our proof of Lemma~\ref{highleafdegreecase} can be adapted to the new setting, with two possible exceptions. 
 
 First, if  $G$ is $\gamma$-special  (see Definition~\ref{gammaspecial} below), our embedding scheme will fail, because the matchings we need for the embedding of~$F_1$ might not exist. In~\cite{brucemaya2} we show how to embed $T$ in that case. Second, if $T$ has a very specific shape, and $G$ is close to containing a complete tripartite graph, we may not be able to find the in-good path partition needed for embedding $T_{NearBal}$. This case is covered by Lemma~\ref{maya200} below, which shows that then we can embed all of $T$.
 
Another important difference to Lemma~\ref{highleafdegreecase} is that in Lemma~\ref{maya1}, we know that no seed of $T$ has many leaves hanging from it. So, we can forget about all the extra work that was done  in the proof of Lemma~\ref{highleafdegreecase} to embed the set $L$ of leaves hanging from seeds.

Let us now give the two definitions we need to state Lemma~\ref{maya1}.

\begin{definition}\label{gammaspecial}
We say a graph $G$ on $m+1$ vertices is $\gamma$-special, for some $\gamma >0$, if $V(G)$ consists of three mutually disjoint sets $X_1,X_2,X_3$ such that
\begin{enumerate}[(i)]
\item  $\frac{m}{3}-{3\gamma} m\le |X_i|\leq\frac{m}{3}+{3\gamma} m$ for each $i=1,2,3$; and
\item there are at most $\gamma^{10} |X_1|\cdot|X_2|$ edges between $X_1$ and $X_2$.
\end{enumerate}
\end{definition}


\begin{definition}\label{gamma-nice}
Let $T$ be a tree with $m$ edges.
Call  a subtree~$T^*$ of  $T$ with root $t^*$ a {\it $\gamma$-nice subtree} if $|V(T^*)|<\gamma m$ and
every component of $T-T^*$  is adjacent to $t^*$.
\end{definition}

We are now ready for the result we will need in~\cite{brucemaya2}.

\begin{lemma}
\label{maya1}
For all $\gamma<\frac 1{10^6}$ 
 there are  $m_0\in\mathbb N$ and $\lambda>0$ such that the following holds for all $m\geq m_0$. \\ Let $G$ be an $(m+1)$-vertex  graph of minimum degree at least $\lfloor \frac{2m}{ 3}  \rfloor$ and with a universal vertex,  such that $G$ is not  $\gamma$-special.  Let~$T$ be a tree with $m$ edges such that $T\not\subseteq G$ and no vertex in $T$ is adjacent to more than $\lambda m$ leaves.
 Let $T^*$ be a $\gamma$-nice subtree of~$T$ with root $t^*$, let $Y\subseteq V(T^*)\setminus\{t^*\}$, and let $S\subseteq V(G)$ with
 $|S|\le |Y|-(\frac{\gamma}2)^4m$. \\ 
 Assume that for any  $W\subseteq V(G)-S$ with $|W|\geq \gamma m$, there is an embedding $\phi_W$ of $T^*-Y$ into $G- S$, with $t^*$ embedded in $W$. Then there is a set $W\subseteq V(G)-S$ with $|W|\geq \gamma m$, and an embedding of $T-Y$  into $G-S$ that extends $\phi_W$. 
\end{lemma}

In the proof of Lemma~\ref{maya1}, we will need the following definition and  lemma. 
\begin{definition}
Call~a~rooted tree {\it bad} if it is a three-vertex path whose root is not the middle vertex.
\end{definition}
\begin{lemma}
\label{maya200}
Let $H$ be a graph on $m+1$ vertices with universal vertex $w$. Let   $H_1\cup\ldots\cup H_5$ be a partition of $V(H)$, with $|H_1|= |H_2|= |H_3|\ge\frac {33}{100}m$  and $|H_5|\le \frac 12 |H_4|$. Assume  that for all $i,j\in\{1,2,3\}$ with $i\neq j$, each vertex of~$H_i$ is adjacent to at least $\frac {99}{100}$ of the vertices in~$H_j$, and for $i=4,5$, each vertex of $H_i$ is adjacent to at least two fifths of the vertices of $H_{i-2}$. 
\\ Let $T$ be a tree with $m$ edges, and let $W, L, F_1, F_2$ be sets as in Lemma~\ref{cutT} for some $\beta$ with $\beta^2 m>250$, except that we do not require the upper bound in Lemma~\ref{cutT}~\eqref{small}. If $F_1$ contains at least $\frac{33}{100}m$  bad trees, then $T\subseteq G$.
\end{lemma}

We leave the proof of Lemma~\ref{maya200} to the end of the section, and first prove Lemma~\ref{maya1}, mainly following the lines of the proof of Lemma~\ref{highleafdegreecase}. 
 
\begin{proof}[Proof of Lemma~\ref{maya1}]
We structure the proof according to the main steps.
\paragraph{\noindent  Setting the constants.}
 Given $\gamma$,   
 we set 
 $\alpha:= (\frac{\gamma}2)^4;$
this will be our approximation factor  for the embedding of $T-Y$. 
We choose  $\eps$ such that 
\begin{equation}\label{gamma_eps}
\eps\le\gamma^{20}.
\end{equation}
Apply  Lemma~\ref{RL}  (the regularity lemma) to $\eps^2$ and $M_0:=\frac 1{\eps^2}$ obtaining num\-bers $M_1$ and~$n_0$. We choose $\beta\ll\eps$, and $\lambda\le\frac{\beta^2\cdot \eps}{3000}.$ Finally, choose a sufficiently large number $m_0$ for the output of Lemma~\ref{maya1}. Resumingly, we have 
\begin{equation}\label{constants2}
\frac 1{m_0}\ll\lambda\ll\beta\ll\eps\ll\alpha\ll\gamma.
\end{equation}
 
 \paragraph{\noindent Holes.}
Call a subset of $V(G)$ containing at least $\frac m3-7\gamma m$ vertices a {\it hole}
 if it induces less than $100\eps m^2$ edges. Let $V_{Bad}$ be the set of all vertices in $V(G)$ whose non-neighbourhood contains a hole. For now, suppose that 
\begin{equation}\label{no-hole-for-now}
\text{$|V_{Bad}|\le \frac m2$.}
\end{equation}
(The other case will be treated at the end.) Now, assume we are given a graph $G$ as in the lemma, a set $S$, a tree~$T$ and  a subtree $T^*$ of $T$, with 
\begin{equation}\label{sizeTstar}
|V(T^*)|<\gamma m,
\end{equation}
 a vertex $t^*\in V(T^*)$, and a set $Y\subseteq V(T^*)\setminus\{t^*\}$. By~\eqref{no-hole-for-now}, we can choose $W:=V(G)\setminus V_{Bad}$. Then there is 
 an 
 embedding $\phi$ of $T^*-Y$  into $G-S$ with 
\begin{equation}\label{outsidehole}
\text{$\varphi (t^*)\notin V_{Bad}$.}
\end{equation}

\paragraph{\noindent Regularising the host graph.}
 We take an $\eps^2$-regular partition of $G':=G\setminus(\varphi (V(T^*-Y)\cup S)$, with a reduced graph $R_{G'}$ on $M_0\le p'\le M_1$ vertices. We wish to extend the embedding of $T^*-Y$ to an embedding of  all of $T-Y$ into $G-S$. 
Note that  the minimum degree of $R_{G'}$ is no longer bounded from below by $(\frac 23-13\eps)p$, as in the proof of Lemma~\ref{highleafdegreecase}, because of the possible degree into the set $\varphi (V(T^*-Y))$. But we can  guarantee the following bound: 

\begin{equation}\label{worsemindeg1}
\delta_w(R_{G'})\geq (\frac 23-\gamma+(\frac{\gamma}2)^4-13\eps)p'\geq (\frac 23-\gamma+\frac{\gamma^4}{20})p'.
\end{equation}

\paragraph{\noindent Cutting the tree.}
We use Lemma~\ref{cutT} to cut up the tree induced by $V(T-T^*)\cup \{t^*\}$, making~$t^*$ a seed. Add all neighbours  of~$t^*$ belonging to a tree from $F_2$ to the set $W$ of seeds. Use Lemma~\ref{cutT}~\eqref{small},~\eqref{oneneighbour}  and~\eqref{atmosttwo} to see that there are at most $2\beta (m-|V(T^*)|)$ new seeds (with exactly the same argument as the one used to prove Lemma~\ref{cutT}~\eqref{eqXX}).
 Note that the new seeds  transform the partition of the tree a little, as any new seed cuts the tree from $F_2$ it belonged to. We just add the newly formed small trees to $L$, $F_1$, $F_2\setminus F_2'$ or $F_2'$, as appropriate, and, slightly abusing notation, continue to call these sets $L$, $F_1$, $F_2\setminus F_2'$ or $F_2'$. Let $W^*$ be the set of all seeds adajcent to $t^*$.
 It will not be necessary to add any more extra seeds, as we did in the proof of Lemma~\ref{highleafdegreecase}, so the total number of seeds is bounded by $\frac 3{\beta^2}$. 
 

\paragraph{\noindent Embedding leaves incident with $t^*$, and reserving for $W^*$.}
First embed the leaves incident with $t^*$, into any  cluster, using the minimum degree of $G$. As $t^*$ has at most $\lambda m\ll \eps m$ leaves hanging from it, this will not disturb the rest of the embedding process.
Let $L'$ denote the set of the remaining leaves from $L$. 

Next, we choose a cluster $C^*$ such that at least a third of its vertices are neighbours of $\varphi (t^*)$. We reserve a set $C^*_W\subseteq C^*$ of  size $\eps^{\frac 13}m$ consisting of neighbours of $\varphi (t^*)$ in $C^*$. This reservation  ensures that we will not block the neighbourhood of $\varphi (t^*)$ before embedding  $W^*$.

\paragraph{\noindent Embedding the trees from $F_1^*$.}
Now we embed the trees from  $F_1^*:=F_1\cap \mathcal T_{t^*}$.  We provisionally slice up each of the clusters of $R_{G'}$ into $\frac 1\eps$ smaller sets (slices) of equal sizes (plus a very small garbage set), in a way that  at most one of the new slices  contains both neighbours  and  non-neighbours of~$\varphi(t^*)$. Let $\mathcal S$ be the set of all these  slices except for  the garbage set and the mixed slice, and let $R'_{G'}$ be the reduced graph on  $\mathcal S$. Say $|\mathcal S|=p''$. Note that $R'_{G'}$ is still regular (with a slightly worse approximation), and in $R'_{G'}$, the minimum degree bound from~\eqref{worsemindeg1} becomes
\begin{equation}\label{worsemindeg}
\delta_w(R'_{G'})\geq (\frac 23-\gamma)p''.
\end{equation}
Consider a set $N\subseteq\mathcal S$ of size $\lfloor (\frac {2}3-\gamma)p''\rfloor$ such that $\varphi({t^*})$ is adjacent to all vertices in all clusters of $N$. We will now  find a matching $M^*$  and path partitions $\mathcal P^*_A$, $\mathcal P^*_B$, or slight variations thereof, as in the proof of Lemma~\ref{highleafdegreecase}, where we employed Lemmas~\ref{lem:match} and~\ref{lem:match_paths}. 

Let us start with the matching $M^*$ from Lemma~\ref{lem:match}. While the conditions of the lemma are still satisfied if we take $\xi$  of the order of $\gamma$, we would like to have an outcome with $\xi$ having the order of $\eps$.
The only possible  reason that could prevent us from finding a set $Y$ of order around $500\eps p''$ and an $N$-good perfect matching $M^*$ of $R'_{G'}-Y$  as in Lemma~\ref{lem:match} is that the first line of the proof of Lemma~\ref{lem:match}, where we greedily match using the minimum degree condition, fails in our new circumstances (for the rest of the argument in that proof  we only need a much weaker minimum degree condition). So, if we cannot find  $M^*$, then 
each matching from $V(R'_{G'})\setminus N$ to~$N$ leaves more than  $\lfloor 500\eps p''\rfloor$ vertices from $V(R'_{G'})\setminus N$ uncovered. 
Thus there is a Hall-obstruction, i.e., a set $X_1'\subseteq V(R'_{G'})\setminus N$ with less than $|X_1'|-\lfloor 500\eps p''\rfloor$ neighbours in $N$. Now, apply~\eqref{worsemindeg} to any vertex from $X_1'$, and set $X_3:=N(X_1')\cap N$ and $X_1:=V(R'_{G'})\setminus N$ to see that 
   \begin{equation}\label{eq:X1}
   |X_1\cup X_3|\geq \left(\frac 23-\gamma\right)p''.
      \end{equation}
Moreover, as $|X_1|=p''-|N|$, we see that
   \begin{equation}\label{S1}
\left(\frac 13+\gamma\right){p''}\geq |X_1| \geq \left(\frac 13-\frac{\gamma}2+250\eps\right)p''\geq \left(\frac 13-\frac{\gamma}2\right)p'',
   \end{equation}
   where for the second inequality we use~\eqref{eq:X1}, as well as the fact that $|X_3|<|X_1'|-\lfloor 500\eps p''\rfloor\le |X_1|-\lfloor 500\eps p''\rfloor$, which implies that $$|X_1|\geq |X_1\cup X_3|-|X_3|\ge \frac {|X_1\cup X_3|}2+250\eps p''.$$
Using~\eqref{eq:X1} and~\eqref{S1}, we can bound the size of $X_3$ as follows:
   \begin{equation}\label{S3}
\left(\frac 13+\gamma\right){p''}\ge |X_1|
> |X_3|=|X_1\cup X_3|-|X_1|\geq \left(\frac 13-2\gamma\right)p''.
   \end{equation}
   By~\eqref{S1} and~\eqref{S3}, and by the choice of $X_1'$, we also have
      \begin{equation}\label{X1'!}
    \left(\frac 13+\gamma\right){p''}\geq |X_1| \geq  |X_1'| \ge |X_3|\geq \left(\frac 13-2\gamma\right)p''.
       \end{equation}
       In particular, the fact that $ |X_1'| \ge |X_3|$ together with~\eqref{S1} ensures that $X_3':=X_3\cup (X_1\setminus X_1')$ has size at most $\left(\frac 13+\gamma\right){p''}$, and hence, by~\eqref{S3}, 
       \begin{equation}\label{X3'!}
    \left(\frac 13+\gamma\right){p''}\geq   |X_3'| \ge |X_3|\geq \left(\frac 13-2\gamma\right)p''.
       \end{equation}
Letting $X_2$ denote the non-neighbours of $X_1'$ in $N$, we obtain from~\eqref{worsemindeg} in a similar way as for $X_1$ that 
  $|X_2\cup X_3|\geq (\frac 23-\gamma )p'',$
      and therefore, using~\eqref{S3} for a bound on $|X_3|$, we obtain that
   \begin{equation}\label{S2}
|X_2|=|X_2\cup X_3|-|X_3|\geq \left(\frac 13-2{\gamma}\right)p''.
\end{equation}
 Also, by~\eqref{eq:X1}, we have that
   \begin{equation}\label{S2'}
 |X_2|\leq \left(\frac 13+\gamma\right){p''}.
\end{equation}
But then $G$ is $\gamma$-special. Indeed,~\eqref{sizeTstar},~\eqref{X1'!},~\eqref{X3'!},~\eqref{S2} and~\eqref{S2'} imply that  Definition~\ref{gammaspecial}~$(i)$ holds for $\bigcup X_1'$, $\bigcup X_2$, and $\bigcup X_3' \cup \varphi (V(T^*)-Y)\cup S\cup (V(G')\setminus \bigcup_{S\in \mathcal S}V(S))$, while Definition~\ref{gammaspecial}~$(ii)$ holds by the definition of~$X_2$, and since by~\eqref{gamma_eps},  any non-edge in $R'_{G'}$ corresponds to a very sparse pair of clusters in $G$. 
However, $G$ being $\gamma$-special is against the assumptions of Lemma~\ref{maya1}. So, no Hall-type obstruction can exist, and we  find the matching~$M^*$ as desired. Set $Q:=V(M^*)\setminus N$.

Next, we turn to the good path partitions $\mathcal P^*_Q$  and $\mathcal P^*_R$  from Lemma~\ref{lem:match_paths}. 
Let us start with 
$\mathcal P^*_Q$. In the proof of Lemma~\ref{lem:match_paths}, we constructed $\mathcal P^*_Q$ using~$M^*$ and an
 auxiliary matching $M^Q$, which, in turn, was obtained as the union of two matchings, the first of which is $\tilde M^Q$, a maximum  matching  inside $V(M^*)\setminus N$. We define $\tilde M^Q$ in the same way here, and set $\tilde Q:=V(M^*)\setminus (N\cup V(\tilde M^Q))$. Note that $|V(\tilde M^Q)|\ge 6\gamma p''$, as otherwise, $\tilde Q$ is a hole, which  is impossible by~\eqref{outsidehole}.
 In the proof of Lemma~\ref{lem:match_paths}, $\tilde M^Q$ is completed to $M^Q$ by matching clusters from $\tilde Q$ to clusters in $N\setminus R$, where $R$ consists of the clusters matched to $\tilde Q$ by $M^*$.  By the maximality of $\tilde M^Q$,  $\tilde Q$ is independent, and  all but at most one of the vertices in $\tilde Q$  see at most one of the endvertices of any edge in $\tilde M^Q$. So, these clusters see at least $|N\setminus R|-\frac{|V(\tilde M^Q)|}2\ge |\tilde Q\setminus V(\tilde M^Q)|$ clusters in $N\setminus R$, implying that we  can match all but at most one vertex of~$\tilde Q$ in $M^Q$, and thus find an $N$-in-good path partition $\mathcal P^*_Q$, as desired. 
 
 Let us now turn to the $N$-out-good path partition $\mathcal P^*_R$  from Lemma~\ref{lem:match_paths}. To find $\mathcal P^*_R$ we employed a maximum matching $\tilde M^R$ inside  $R$. If $|V(\tilde M^R)|>6\gamma p''$,  we can proceed as
above to find $\mathcal P^*_R$. Otherwise,  it is not hard to see that we can match  all except at most $6\gamma p''$ clusters of $R$ to $N\setminus R$, which leads to almost all thus obtained paths having 6 vertices. We can then match the remaining $6\gamma p''$ clusters of $R$ to the innermost vertices of these paths. In this way,
  we find a partition $\mathcal P^*_{mod}$ of $R$ into vertex-disjoint copies of graphs of the following types: any of the three types of paths allowed in an $N$-out-good path partition, plus the  graph $G_{good}$, which is defined by $V(G_{good})=\{A, A', B, B', C, D, E, E', F, F'\}$, with $B, B', C, D, E, E'\in N$,  and $E(G_{good})=\{AB$, $A'B'$, $BC$, $B'C$, $CD$, $DE$, $DE'$, $EF$, $E'F'\}$. 
 
We embed~$\bigcup F_1^*$ using $M^*$, $\mathcal P^*_Q$  and $\mathcal P^*_{mod}$, all the time avoiding  $C^*_W$. Note that when using $\mathcal P^*_{mod}$ we will have to 
adapt
 our strategy from Section~\ref{sec:tree_embed}, 
 but this is not 
 hard\footnote{For any graph of the new type, consider embedding $x$  trees of type~$\bar T$ in the edge~$AB$, putting the bipartition class $\bar T_1$ that contains the root $r_{\bar T}$ into $B$ and the other class $\bar T_2$ into~$A$. Then  embed  $x$ trees similarly into each of $A'B'$,  $EF$ and  $E'F'$. Next, do~$y$ rounds of embedding $\bar T$  with $r_{\bar T}$ going into $C$, $\bar T_2$ going to $B$, and $\bar T_1-r_{\bar T}$ going to $A$, and $y$ analogous rounds for each of the edges $A'B'$,  $EF$, and $E'F$. Finally, embed $\bar T$ $2z$ times  into $CD$, of which $z$ rounds in each `direction'. After these $4x+4y+2z$ rounds, we filled each of   $A$, $A'$, $F$, $F'$ with 
$xt_2+y(t_1-1)$ vertices, each of  $B$, $B'$, $E$, $E'$ with 
$xt_1+yt_2$ vertices, and each of  $C$, $D$ with 
$2y+zt$ vertices, where $t_i=|\bar T_i|$ for $i=1,2$.
So,  taking $x=(t_1+t_2)\cdot(t_2-t_1+1)$, $y=(t_1+t_2)\cdot(t_2-t_1)$, and $z=(t_1+t_2-2)\cdot(t_2-t_1)+t_1$, we fill each of the eight clusters with exactly the same amount of vertices. Similarly  as in Section~\ref{sec:tree_embed}, we can calculate how to fill all types of graphs from $\mathcal P^*_{mod}$ simultaneously with the same amount of vertices.
}. 
 In order to leave all clusters balanced during the embedding of $\bigcup F^*_1$, we  again use a small set $U'$ of pseudo-used vertices.

\paragraph{\noindent Slicing up the clusters.}
We now go back to work in $R_{G'}$. We slice up the yet unused parts of the clusters as before (but avoiding $C^*_W$), into sets $C_L$, $C_W$, $C_{\tilde V}$, $C_{F_1\setminus F_1^*}$, and $C_{F_2}$. The slices $C_X$ reflect the sizes of the corresponding sets $X$, but we leave sufficient buffer space in each.  

We now go through the subtree induced by $W$  and the non-trivial trees hanging from them in a connected way, starting with the root~$t^*$. As before, we embed each seed together with all small trees from $F_1\cup F_2$ hanging from it.
We always avoid the set $C^*_W$, unless we are  embedding a seed from $W^*$.

\paragraph{\noindent Embedding a seed $s$.} We embed each seed $s$  in a neighbour $\phi(s)$ of the image $\varphi (p)$ of its parent $p$, with $\phi(s)\notin V_{Bad}$, and such that $\phi(s)$ is  typical with respect to the slices $C_L$, $C_{W}$, $C_{\tilde V}$ and $C_{F_1\setminus F_1^*}$. Note that this is possible as  by~\eqref{no-hole-for-now}, and by our condition on the minimum degree of $G$, the vertex $\phi(p)$ has plenty of neighbours outside $V_{Bad}$. Usually, seeds go to $C_{W}$, but seeds  from $W^*$ go to their reserved space $C^*_W$.
Note that we did not need to group and order our seeds as in  the proof of Lemma~\ref{highleafdegreecase}, and we also do not need to choose the target clusters as carefully.

\paragraph{\noindent Embedding the trees from $F_2\cup F_1$ at $s$.} For the trees from $F_2$, we can find the matching $M_{F_2}$ just as before, as the minimum degree bound~\eqref{worsemindeg} is sufficient. Again, we make the connections through the slices $C_{\tilde V}$, using the fact that the corresponding seeds are embedded in  typical vertices (with respect to almost all slices $C_{\tilde V}$).

For the trees from $F_1$, we proceed as above for $F^*_1$, i.e~we slice up the clusters, so that almost all of the obtained slices behave uniformly with repect to being adjacent  to  $\phi(s)$. Let  $R''_{G'}$  be the graph on these slices, after momentarily discarding the mixed slice and the garbage slice. 
Since $G$ is not $\gamma$-special, and since we embedded $s$ outside any hole, we can proceed exactly as above (when we embedded the trees from $F_1^*$) to find a matching $M$, an in-good path partition and a modified out-good path partition in $R''_{G'}$. 

\paragraph{\noindent Embedding the leaves from $L'$.}
In~\eqref{constants2}, we chose $\lambda$ such that $$|L'|\ \leq\ \frac 3{\beta^2}\cdot \lambda m\ \leq \ \frac 1{1000}\eps m.$$
Since there are about $p'$ slices $C_L$, each much larger than $\eps \frac m{p'}$, and since the embedded seeds (except possibly $t^*$) see about two thirds of almost all  slices $C_L$, we can embed the leaves from $L'$ greedily into $\bigcup_{C\in V(R_{G'})} C_L$.
This finishes the proof of Lemma~\ref{maya1} for the case that~\eqref{no-hole-for-now} holds.

\paragraph{\noindent More holes.}
Now assume that~\eqref{no-hole-for-now} does not hold. 
We can delete a few vertices from a hole so that inside the remaining set, no vertex has degree more than $\sqrt\eps m$.
By deleting a few more vertices, we arrive 
at a hole of size exactly $\lceil(\frac 13-8\gamma)m\rceil$. Let $\mathcal H$ be the set of all holes obtained in this way. As~$G$ has  minimum degree at least $\lfloor \frac {2m}3\rfloor$, for each $H\in\mathcal H$ and $x\in V(H)$ we have
\begin{equation}\label{degree_in_holes}
\deg(x, V(G)\setminus V(H))\ge |V(G)\setminus V(H)|-9\gamma m.
\end{equation}
In particular, since~\eqref{no-hole-for-now} is not true, $\mathcal H$ contains at least two holes, $H_1$ and~$H_2$, and  their intersection  is empty.

By~\eqref{degree_in_holes} it is easy to calculate that for $i=1,2$, all but at most $3\sqrt\gamma m$ vertices of $V(G)\setminus(H_1\cup H_2)$ have degree at least $|H_i|-\sqrt\gamma m$ into $H_i$.
We choose a set $H_3\subseteq V(G)\setminus(H_1\cup H_2)$ with $|H_3|=|H_1|=|H_2|$ and such that for each $x\in H_3$, and each $i\in\{1,2\}$,
$\deg(x, H_i)\ge |H_i|-\sqrt\gamma m.$

Set $H_0:=V(G)\setminus (H_1\cup H_2\cup H_3)$. As $\delta (G)\ge \lfloor \frac {2m}3\rfloor$, each vertex in $H_0$ sees at least two fifths the vertices  of at least two of the sets $H_1$, $H_2$, $H_3$. So, there is an index $i\in\{1,2,3\}$, without loss of generality let us assume   $i=2$, such that
at least two thirds of the vertices in $H_0$ are adjacent to at least two fifths of the vertices in $H_2$. In other words, we can split $H_0$ into two sets $H_4$, $H_5$ as in Lemma~\ref{maya200}.

Let $T_{Bad}$ be the set of all bad trees in $F_1$. If
 $|T_{Bad}|\ge\frac{33}{100} m$, we  add the trees from $T^*-t^*$ to $L$, $T_{Bad}$, $F_1\setminus T_{Bad}$, and $F_2$, where we now  allow that trees in~$F_2$  have more than $\beta m$ vertices.  Forgetting about the embedding~$\phi$ of $T^*$, we apply Lemma~\ref{maya200} to $G$, finding that $T\subseteq G$, contrary to the assumptions of Lemma~\ref{maya1}. So, 
 $$|T_{Bad}|\le\frac{33}{100} m.$$
Now, we let $t^*$ be embedded into the hole $H_1$. We regularise $G-\phi(T^*-Y)$, respecting the prepartition given by $H_1, H_2, V(G)\setminus (H_1\cup H_2)$.
In the reduced graph $R_G$ on $p$ clusters, we use~\eqref{degree_in_holes} to find an edge $X_1X_2$ with $X_i\subseteq H_i$ for $i=1,2$, and $|X_2\cap N(\phi(t^*))|\ge \frac{|X_2|}2$. Using~\eqref{degree_in_holes} again, we find a set $\mathcal D$ of disjoint triangles $ABC$ in $R$ with $A\subseteq H_1$, $B\subseteq H_2$ and $C\subseteq  H_3$,  such that for each  $i=1,2$, two of the  pairs $(A,X_i)$, $(B, X_i)$, $(C,X_i)$ have large density (larger than $\frac 23$), and such that moreover, $\phi(t^*)$ sees at least $\frac 23$ of the vertices in $B$ and in $C$.
 We can choose  $\mathcal D$ such that  $|\mathcal D|=\lceil(\frac{1}3-100\gamma )p\rceil$.  
 
We now split the clusters from $V(R_G)\setminus\{X_1, X_2\}$ into slices. Each cluster belonging to a triangle from $\mathcal D$ is split into one slice $D_1$ of size $\frac{|T_{Bad}|}{|\mathcal D|}+100\eps m$, and one slice $D_2$ containing the remaining vertices. Doing this, we put as many neighbours of $\phi(t^*)$ as possible into $D_1$. Let $\mathcal S$ be the set of all slices of type $D_1$ and note that together they are large enough to accommodate all trees from $T_{Bad}$ (there is even some buffer space). Let $\mathcal D'$ denote the set of triangles in the reduced graph on $\mathcal S$ corresponding to triangles from~$\mathcal D$. Now, take the set of all slices of  type $D_2$, and all clusters in $V(R_G)\setminus(V(\mathcal D)\cup\{X_1, X_2\})$, and split each of these elements into slices of size $\eps\frac mp$ (plus possibly one garbage set). Let $\mathcal S'$ be the set of all these slices, and denote by~$R'_G$ the reduced graph on $\mathcal S'$. By our bound on $|T_{Bad}|$, we know that 
\begin{equation}\label{minimini}
\delta(R'_G)\ge(\frac 23-500\gamma)p',
\end{equation}
 where $p'=|\mathcal S'|$.
Now, we embed the leaves hanging from $t^*$ as before, and
 embed the bad trees at $t^*$ into  $\bigcup\mathcal S'$,  using regularity inside the triangles, and filling the clusters from $\mathcal S'$ as evenly as possible. 
 For this, consider a triangle $A'B'C'\in\mathcal D'$ and note that $\phi(t^*)$  sees almost $\frac 23$ of  the clusters $B'$ and $C'$. So, if necessary we can fill almost all  $A'\cup B'\cup C'$ with trees  from $T_{Bad}$, by distributing their first vertices among $B'$ and $C'$.

We then embed the trees from $F_1\setminus T_{Bad}$ adjacent to~$t^*$. For this, we temporarily slice up the clusters from $\mathcal S'$ into about $\frac 1{\eps}$ new slices, so that almost all new slices contain either only neighbours or only non-neighbours of $\phi(t^*)$, disregarding the possible garbage slice and the possible mixed slice, as before.  Call this new set of slices $\mathcal S'_{t^*}$. 

In the reduced graph on $\mathcal S'_{t^*}$, we find  a matching~$M^*$ (for this, we use~\eqref{minimini} and the fact that most of the vertices in the non-neighbourhood of $\phi(t^*)$ lie in $H_1$, which means that they see almost all of $V(G)\setminus H_1$).
 We then find a modified out-path partition $\mathcal P_{mod}$ as before. We may be unable to find 
a in-good path partition $\mathcal P_Q$ as before (as now $\phi(t^*)$ lies in a hole). But, in a similar way as we found  $\mathcal P_{mod}$, we can find a modified partition $\mathcal P'_{mod}$, which allows for graphs of the three types from the definition of the  in-good path partition, plus  a new type of graph on eight vertices $A, A', B, B', C, D, E, E', F, F'$, with $A, A', C, D, F, F'\in N$,  and edges $AB$, $A'B'$, $BC$, $B'C$, $CD$, $DE$, $DE'$, $EF$, $E'F'$. 
Recall that the in-good path partition was only used for the trees from  $T_{NearBal}$.
It is not hard to see that all trees in $T_{NearBal}\setminus T_{Bad}$ can be embedded into the 
modified in-good path 
partition
 $\mathcal P'_{mod}$.\footnote{In order to show this, let us just prove that any graph of the new type can be equally filled with a fixed type of tree $\bar T\in T_{NearBal}\setminus T_{Bad}$. For this, let $t_1$ denote the size of the larger partition class of $\bar T$. Note that $t_1\ge 3$ and the other class has size $t_1-1$. Let $r_{\bar T}$ be the root of $\bar T$. Embed $2t_1-1$ times $r_{\bar T}$  into $C$,  the smaller bipartition class into $B$, and  the rest into $A$. Do the same three more times, replacing the triple $C$, $B$, $A$ with $C$, $B'$, $A'$, with $D$, $E$, $F$, and with $D$, $E'$, $F'$.
Embed $t_1-3$ times  the larger bipartition class into $C$, and  the smaller  class into $D$. Also embed  $t_1-3$ copies of $\bar T$  the other way around into $D$ and $C$. Then each of the eight clusters is filled with exactly $(2t_1-1)(t_1-1)$ vertices.
}
We fill all clusters from $\bigcup\mathcal S''$  evenly (again using a set $U'$).  

Next, we slice up the yet unused parts of the clusters as before  into sets $C_L$, $C_W$, $C_{\tilde V}$, $C_{F_1\setminus F_1^*}$, $C_{F_2}$, each of the size needed, plus buffer space. We  go through  the seeds~$s$ and the trees from $F_1\cup F_2$ hanging from $s$. We embed~$s$  into a vertex $\phi(s)\in X_i$, for some $i=1,2$, such that $\phi(s)$ is typical with respect to each of the clusters from triangles $A'B'C'\in\mathcal D'$, with respect to the unused parts of these clusters, with respect to the unused parts of slices of clusters from $\mathcal S'$, and with respect to $X_{3-i}$. We also require that $s$ is embedded into a neighbour of the image of its parent, which is not a problem, since all such parents are embedded into vertices having sufficient degree into $X_1\cup X_2$. 

We embed all trees  from $T_{Bad}$ hanging from $s$ into  $\bigcup\mathcal S$, using regularity and filling all clusters from $\mathcal S$ almost evenly.  
Note that $\phi(s)$ is typical with respect to the unused part of at least two of the clusters of any
 triangle $A'B'C'\in\mathcal D'$, i.e.~$\phi(s)$  sees almost $\frac 23$ of these unused parts. So, as above for $\phi(t^*)$, we can distribute the first vertices of the trees in $T_{Bad}$ at $s$
  among the two clusters  and thus fill up the triangle  if necessary.

We then embed the trees from $F_1\setminus T_{Bad}$ at $s$  into $\bigcup\mathcal S'$, using  the same strategy as before, i.e.~temporarily slicing up the clusters from $\mathcal S'$ into about $\frac 1{\eps}$ new slices, so that almost all contain either only neighbours or only non-neighbours of $\phi(s)$. Finding  a matching~$M^*$ and modified path partitions as above, we can  fill all clusters from $\bigcup\mathcal S''$  evenly.  
For the trees in $F_2$ hanging from $s$, we use a matching $M_{F_2}$, as before. If $\bar T\in F_2$ contains a parent $p$ of a seed, we  embed $p$ into a suitable vertex $v\in C_{\tilde V}$ with many neighbours in $X_1\cup X_2$.
We finish by embedding the  leaves hanging from seeds as before.
\end{proof}

It remains to prove Lemma~\ref{maya200}. For its proof it will be convenient to have the following auxiliary result at hand. 

\begin{lemma}
\label{maya100}
Let $H$ be a graph with universal vertex $w$. Let   $H_1, H_2, H_3\subseteq V(H)$ be disjoint, with $|H_i|\ge\frac {3}{10}m$ and such that
each vertex of~$H_i$ is adjacent to at least $\frac {9}{10}$ of the vertices in~$H_j$, for $i=1,2,3$ and $j\in\{1,2,3\}\setminus\{i\}$. Let~$T$ be a tree on $m+1$ vertices, with sets $W, L, F_1, F_2, F'_2$  as in Lemma~\ref{cutT} for some $\beta<1$ with $\beta^2 m>600$, except that we do not require the upper bound in Lemma~\ref{cutT}~\eqref{small}. If~$F_1$ contains at least $\frac{33}{100}m$ bad trees,
then there is a tree $T'\subseteq T$ with $|V(T')|\le \frac m{50}$ and an embedding $\phi$ of $T'$ into $H$ such that 
\begin{enumerate}[(i)]
 \item each component of $T-T'$ is a bad tree from $F_1$, and for at least $\frac{33}{400}m$ of these trees, their neighbour in $\phi(T')$ is embedded in  $H_2\cup\{w\}$; and\label{maya100seeds}
\item $|\phi(V(T'))\cap H_2|\ge |\phi(V(T'))\cap H_i|$ for $i=1,3$.\label{maya100stuffed}
\end{enumerate}
\end{lemma}
\begin{proof}
{\it Step 1: Preparation.}
Let $T_{Bad}$ be the set of all bad trees in $F_1$ and let $T''$ be the union of~$W$ and all trees from $(L\cup F_1\cup F_2)\setminus T_{Bad}$.  Note that  by the assumptions of the lemma, 
\begin{equation}\label{le_tigre}
|T''|\le \frac m{100}.
\end{equation}
Consider the bipartition $W_0\cup W_1$ of~$W$  induced by the bipartition of~$T$  and say~$W_0$ is adjacent to at least as many  trees from $T_{Bad}$  as $W_1$. So,  there are at least $\frac {33}{200} m$ bad trees hanging from vertices of~$W_0$.

Partition~$W_0$ into three sets $W_0'$, $W'_2$, $W'_3$ such that~$W'_0$ contains all seeds from~$W_0$ not adjacent to any leaves, and such that  $|(\ell_{W'_3}+|W'_2|)-(\ell_{W'_2}+|W'_3|)|$ is minimised, where we define $\ell_X$ as the number of leaves adjacent to the vertices of $X$, for any $X\subseteq W$.
 Say $W'_2$ is adjacent to at least as many bad trees as $W'_3$. 
 If $\ell_{W'_3}+|W'_2| \ge \ell_{W'_2}+|W'_3|$, choose $s^*$  arbitrarily  in $W'_3$, and if $\ell_{W'_3}+|W'_2| < \ell_{W'_2}+|W'_3|$, there is at least one vertex in $W_2'$ adjacent to more than one leaf, let $s^*$ be such a vertex.
Then, set $W_2:=(W'_0\cup W'_2)\setminus\{s^*\}$ and $W_3:=W'_3\setminus\{s^*\}$, and note that
 by construction,
\begin{equation}\label{realnice2}
 \text{there are at least $\frac {33}{400}m$ bad trees hanging from $W_2\cup \{s^*\}$, }
\end{equation}
\begin{equation}\label{adjacentseeds}
 \text{$W_1$ and $W_2\cup W_3$ belong to different bipartition classes of $T$,}
\end{equation}
\begin{equation}\label{ineedaleaf}
 \text{every seed in $W_3$ is adjacent to at least one leaf, and}
\end{equation}
\begin{equation}\label{realnice3}
 \text{$\ell_{W_3}+\ell_{\{s^*\}}+|W_2|\ge \ell_{W_2}+|W_3|$.}
 \end{equation}
 Indeed, \eqref{realnice3} is obvious if $s^*\in W'_3$, and if $s^*\in W'_2$, note that  by our choice of $W'_2$ and $W'_3$, we have $$\ell_{W'_3}+\ell_{\{s^*\}}+|W'_2\setminus\{s^*\}| > \ell_{W'_2}-\ell_{\{s^*\}}+|W'_3\cup\{s^*\}|,$$ which implies \eqref{realnice3}.

Finally, for each $s\in W_1$ that is adjacent to a tree from $T_{Bad}$, choose one such tree, and let $T_{W_1}$ be the set of all these trees. Let $T'$ be the tree induced by $T''\cup T_{W_1}$. Note that since $|V(T_{W_1})|\le \frac 6{\beta^2}<\frac m{100}$ by Lemma~\ref{cutT}~\eqref{few}, and because of~\eqref{le_tigre}, we have that
\begin{equation}\label{leonard}
|T'|\le \frac m{50}.
\end{equation}

{\it Step 2: Embedding $T'\setminus L$.}
Now, we will go through all seeds  $s\in W$, starting with any seed and then always choosing a seed whose parent is already embedded. We embed the seed $s$, and all trees from $\bar T\in T_{W_1}\cup (F_1\setminus (T_{Bad}))\cup F_2$ that are hanging from $s$. (The set $L$ will be embedded in the next step.)

For each embedded seed, we will ensure that
\begin{enumerate}[(a)]
\item \label{planets} if $s\in W_i$ then $\phi(s)\in H_i$ for $i=1,2,3$, and $\phi(s^*)=w$.
\end{enumerate}
For each tree $\bar T$ embedded in this step, the following will hold:
\begin{enumerate}[(a)]\setcounter{enumi}{1}
\item \label{balan2}
$|\phi(V(\bar T))\cap H_2|\ge \max\{|\phi(V(\bar T))\cap H_3|,|\phi(V(\bar T)\cup s)\cap H_1|\}$, and
\item if $\bar T$ contains a parent $p$ of a seed $s'\in W_{i'}$, then $\phi(p)\notin H_{i'}$.\label{GoodParent}
\end{enumerate}

Say in the current step we are dealing with $s\in W$. If $s=s^*$, we embed $s$   into $w$, which is possible  since $w$ is universal. If $s\in W_i$, then we embed $s$ into $H_i$. Note that this is possible because  of  the minimum degree  between the sets $H_i$, and  by~\eqref{planets} and~\eqref{adjacentseeds} if the parent is itself a seed, or by~\eqref{GoodParent}  if the parent is not a seed. So,~\eqref{planets} holds for $s$. In what follows, our arguments will often use  the minimum degree  between the sets $H_i$ without explicitly mentioning it.

Now we go through the trees 
 $\bar T\in T_{W_1}\cup (F_1\setminus (T_{Bad}))\cup F_2$  hanging from~$s$, in any order, and embed each of them. Say we are at tree $\bar T$, 
 with root $r_{\bar T}$. We need to show we can embed $\bar T$ such that $(b)$ and $(c)$ hold. For this, we distinguish between the types of trees in $T_{W_1}\cup (F_1\setminus (T_{Bad}))\cup F_2$.

{\it Case 1: $\bar T\in T_{W_1}$.} In this case,~\eqref{planets} ensures that $\phi(s)\in H_1$, and so, we can embed the first and the last vertex of $\bar T$ into $H_2$, and the middle vertex into $H_1$, which ensures~\eqref{balan2}. 
Note that $\bar T$ contains no parents of seeds, and so~\eqref{GoodParent} is void.

{\it Case 2: $\bar T\in (F_1\setminus (T_{Bad})) \cup (F_2\setminus F'_2)$.} 
Let $\bar T_1, \bar T_2$ be the partition of $V(\bar T-r_{\bar T})$ induced by the bipartition classes of $\bar T$. We can assume that  $|\bar T_2|\ge |\bar T_1|$. Additionally, if $|\bar T_2|= |\bar T_1|$, 
assume $\bar T_2$ is the class containing the neighbours of $r_{\bar T}$, and  choose any vertex $v\in \bar T_1$. 

We embed $r_{\bar T}$ in $H_1$ unless $\varphi(s)\in H_1$, in which case we embed $r_{\bar T}$ in $H_3$. We 
embed  $\bar T_2$ into $H_2$, and embed $\bar T_1$ into  $H_3$,  unless $\varphi(s)\in H_1$, in which case  we embed $\bar T_1$ into  $H_1$.
Note that then~\eqref{balan2} holds for $\bar T$, unless  $|\bar T_1|= |\bar T_2|$  and $\phi(s)\in H_1$. In that case we move $\phi (v)$ from $H_1$ to $H_3$, which 
ensures~\eqref{balan2}. Note that $\bar T$ contains no parents of seeds, and so, again,~\eqref{GoodParent} is void.

{\it Case 3: $\bar T\in F'_2$.} 
In this case, $\bar T$ contains exactly one parent $p$ of some seed $s'$, say  $s'\in W_{i'}$. We choose $\bar T_1, \bar T_2$ as in case 2, and if $|\bar T_2|= |\bar T_1|$  we make sure that $v\neq p$
(this is possible as $\bar T\notin  T_{Bad}$ and hence if $|\bar T_2|= |\bar T_1|$, then $|\bar T_1|\ge 2$).

We first try to embed $\bar T$ as in case 2. This is successful (i.e.~both $(b)$ and $(c)$ hold) unless $\varphi(p)\in H_{i'}$, so let us assume this is the case. There are now three possibilities: $p=r_{\bar T}$, $p\in \bar T_1$, $p\in \bar T_2$, which we will treat separately.

{\it Case 3a: $p=r_{\bar T}$.} 
In this case, by~\eqref{adjacentseeds}, seeds $s$ and $s'$ belong to the same bipartition class of $T$, that is, either both are in $W_1$, or both are in $W_2\cup W_3$. Recall that in the latter case, the embedding from case 2 maps $r_{\bar T}$ to $H_1$. So it is impossible that $\varphi(p)\in H_{i'}$, a contradiction.

{\it Case 3b: $p\in \bar T_1$.} 
In this case, we relocate the image of $p$, moving $\varphi(p)$ either from $H_1$ to $H_3$, or from $H_3$ to $H_1$. If $\phi (v)$ was moved from $H_1$ to $H_3$, we move it back. 
Now we have an embedding of $\bar T$ for which $(b)$ and $(c)$ hold (note that $\bar T_2$ contains at least two vertices).

{\it Case 3c: $p\in \bar T_2$.} 
First assume that 
$s$ and $s'$ lie in distinct bipartition classes of~$T$. Then, as by assumption $\phi(p)\in H_2$ and $s'\in W_2$, we have that $s\in W_1$, and therefore, by $(a)$, $\phi(s)\in H_1$. If $|\bar T_2|\le|\bar T_1|+1$, we can embed all of $\bar T_1\cup\{r_{\bar T}\}$ in $H_2$, and all of $\bar T_2$ in $H_3$. If $|\bar T_2|\ge |\bar T_1|+2$, we can embed $p$ in $H_1$, $\bar T_2\setminus \{p\}$ in $H_2$, and $\bar T_1\cup\{r_{\bar T}\}$ in $H_3$. In either case we have found an embedding fulfilling $(b)$ and $(c)$.

So we can assume that $s$ and $s'$ are in the same bipartition class of~$T$. Then $s\in W_2\cup W_3$, and therefore, by $(a)$, $\phi(s)\in H_2 \cup H_3$. Also, $s$ and $s'$ being in the same bipartition class of~$T$ means that $r_{\bar T}$ and $p$ are not adjacent. So we can take the embedding from case 2, and move $\phi (p)$ from $H_2$ to $H_1$. This gives an embedding fulfilling $(b)$ and $(c)$.

This finishes the embedding of $T'\setminus L$.

{\it Step 3: Embedding $L$.}

Now consider any $\bar T\in L$, say $\bar T$ hangs from seed $s$. If $s\in W_1\cup W_3\cup\{s^*\}$, we  embed $\bar T$ into~$H_2$, and if $s\in W_2$, we embed $\bar T$ into  $H_3$.

{\it Step 4: Verifying (i) and (ii).}
Note that by $(a)$ and by~\eqref{realnice3},  $|\phi(W\cup L)\cap H_2|\ge |\phi(W\cup L)\cap H_3|$, and thus, by $(b)$, $|\phi(V(T'))\cap H_2|\ge |\phi(V(T'))\cap H_3|$. We will show below that 
\begin{equation}\label{muuu}
|\phi(V(T'))\cap H_2|\ge |\phi(V(T'))\cap H_1|.
\end{equation}
 Then,  by~\eqref{realnice2}, and by~\eqref{leonard}, we have embedded $T'$  in a way that both~\eqref{maya100seeds} and~\eqref{maya100stuffed} hold.
 
It only remains to see~\eqref{muuu}. For this, first note that no $\bar T\in L$ was embedded into $H_1$. Furthermore, for each small tree $\bar T$ from $T_{W_1}\cup (F_1\setminus (T_{Bad}))\cup F_2$, we know the following because of~$(b)$: If $\bar T$  is hanging from a seed $s\in W_1$,  then 
$|\phi(V(\bar T))\cap H_2|\ge |\phi(V(\bar T)\cup s)\cap H_1|$ (observe that the seed $s$ is included on the right hand side); and if $\bar T$  is hanging from a seed $s\in W_2\cup W_3$,  then 
$|\phi(V(\bar T))\cap H_2|\ge |\phi(V(\bar T))\cap H_1|$. 
So, we always embed at least as much into $H_2$ as into $H_1$, if we, for the moment, disregard possible seeds in $W_1$ that have no trees hanging from them. This is still true if we also disregard $W_2$ and all of $L$.

 In other words, letting $S_0$ denote the set of all seeds in $W_1$ having no trees hanging from them, and letting $L_3\subseteq L$ denote the set of all leaves hanging from vertices of $W_3$, we know that 
$$|\phi(V(T')\setminus (W_2\cup L_3))\cap H_2|\ge |\phi(V(T')\setminus S_0)\cap H_1|.$$

We finish the proof of~\eqref{muuu} by showing that $$|\phi(W_2\cup L_3)\cap H_2|\ge |\phi(S_0)\cap H_1|.$$
For this, consider any  seed $s\in S_0$. By Lemma~\ref{cutT}~\eqref{noleafseed}, $s$ has  a child $s'\in W$. As $s\in W_1$, we know that~$s'\in W_2\cup W_3$. If $s'\in W_2$, then $\phi(s')\in H_2$, which accounts for $s$, and if $s'\in W_3$, then by~\eqref{ineedaleaf}, $s'$ has a leaf, which was embedded in $H_2$ and which accounts for $s$.  Noting that distint seeds $s$ have distinct children $s'$, we are done.
\end{proof}
We are now ready to prove the final ingredient for the proof, Lemma~\ref{maya200}.
\begin{proof}[Proof of Lemma~\ref{maya200}]
We start by applying Lemma~\ref{maya100} to find a tree $T'$ with $|V(T')|\le \frac m{50}$, and an embedding $\phi$ of it into $H_1\cup H_2\cup H_3\cup\{w\}$ with the properties given in the lemma.
In particular, there are at least $\frac{33}{400}m$ unembedded  bad trees whose seeds  were embedded into $H_2$.
We extend $\phi$ by embedding some of these trees, in a way that we fill all of $H_4\cup H_5$.

 By the assumptions of Lemma~\ref{maya200}, (and using the minimum degree between the sets $H_i$, and the fact that both $T'$ and $H_4\cup H_5$ are small,) we know that for any $x\in V(H_5)$ we can embed any bad tree~$\bar T$ hanging from $x$  mapping 
its first vertex into $H_1$, its second vertex into an appropriate vertex of  $H_3$, and its third vertex into~$x$. 
For any $x\in V(H_4)$, we can put 
the first vertex of $\bar T$ into either $H_1$ or $H_3$, the second vertex into  $H_2$, and the third vertex into~$x$. 
Alternating  $H_1$ and $H_3$ when filling $H_4$, and recalling Lemma~\ref{maya100}~\eqref{maya100stuffed}, we  can ensure that for $i=1,3$, 
 \begin{equation}\label{realrealnice}
\text{$|\phi(T)\cap H_i|\le |\phi(T)\cap H_2|\le  |V(T')| + |H_4\cup H_5|\le \frac{3}{100}m$, and}
\end{equation}
\begin{equation}\label{stillbadtreees}
\text{$|T_{H_2}|\ge\frac{29}{400}m$,}
\end{equation}
where $T_{H_2}$ is the set of all yet unembedded trees whose seed is embedded in~$H_2$. 
We now embed some more trees from $\bar T\in T_{H_2}$, in two phases. This time our aim is to balance the sizes of the used parts of the three sets $H_1$, $H_2$, $H_3$. 

Let~$H_i$ be the set containing the least number of embedded vertices. By~\eqref{realrealnice}, we know that $i\neq 2$. In the first phase, if $H_i$ has strictly less used vertices than~$H_2$, we embed trees from $T_{H_2}$ by putting their first vertex into $H_i$, their second vertex into $H_2$, and their third vertex into $H_i$. We do this until the used parts of $H_1$ and $H_3$ differ by at most one vertex. 

In the second phase, we embed  trees from $T_{H_2}$, only using  $H_1$ and $H_3$, and keeping their used parts almost perfectly balanced, until the one of them has exactly as many used vertices as $H_2$. 
Then, since the unused part of $G$ is divisible by three (since the unembedded part of $T$ is divisible by three), all three of $H_1$, $H_2$, $H_3$  have exactly the same number of used vertices. By~\eqref{realrealnice}, and since the used part of $H_2$ has augmented by a factor of at most $\frac 32$, this number is at most $\frac{m}{20}$.

Moreover, we used at most $\frac 3{200}m$ trees in the first phase, and at most $\frac 3{200}m$ trees in the second phase.
So by~\eqref{stillbadtreees}, we still have at least $\frac{m}{25}+2$ yet unembedded  trees in $T_{H_2}$. We embed up to two more of these trees, keeping the sets $H_i$ balanced, so that the number of unembedded trees in $T_{H_2}$ is now divisible by~$3$.

We next embed all yet unembedded bad trees from $T_{Bad}\setminus T_{H_2}$, distributing their three vertices equally among the three sets $H_i$. Note that the minimum degree between the sets $H_i$ is large enough so that it does not matter that in this step, all but $\frac{m}{25}$ vertices of each $H_i$ may become occupied.

 In order to embed the remaining trees, we will still  use the minimum degree between the sets $H_i$ to embed their first and second vertices, and then a Hall-type argument to embed their last vertex. Let us make this more precise. We always embed three paths at a time. The first vertex of the first path goes to $H_3$, and the first vertex of each of the other two paths goes to~$H_1$. The second vertex of the second path goes to $H_3$, and the second vertex of the other two paths goes to $H_2$. Doing this for all remaining trees, the unused space $U_i$ in each $H_i$ has the same size, which is at least $\frac m{75}$. 

Our plan is to embed the third vertex $v_P$ of each path $P$ as follows: $v_P$ goes into $H_1$ if $P$ was a `first path', into $H_2$ if $P$ was a `second path' and into $H_3$ if $P$ was a `third path'. 
 Let $A_i$ be the set of all images of second vertices on paths whose third vertex is scheduled to go to $H_i$.
Using Hall's theorem
we see that
if we cannot embed all third vertices as planned, then there is an obstruction, that is, there is an index $i\in\{1,2,3\}$ and a set $A'\subseteq A_i$ such that $|N(A')\cap U_i|<|A'|$. 
 In particular, there is a vertex $v\in U_i\setminus N(A')$, and by our minimum degree condition, $v$ is adjacent to all but at most $\frac m{250}$ vertices of $A_i$. Hence $|A'|\le \frac m{250}\le \frac{|A_i|}2$. On the other hand, any vertex $a\in A'$ is adjacent to all but at most $\frac m{250}$ vertices of $U_i$, and therefore $$|A'|> |N(A')\cap U_i|\ge \frac m{75}-\frac m{250}\ge \frac{|A_i|}2.$$ This  finishes the embedding of $T$.
\end{proof}

\newcommand{\etalchar}[1]{$^{#1}$}

\end{document}